\documentclass[12pt]{amsart}
\usepackage{amssymb,pb-diagram}
\usepackage{tikz}
\usetikzlibrary{cd}
\usetikzlibrary{matrix,arrows,decorations.pathmorphing}

\title{An unstable change of rings for Morava E-theory}

\author{Robert Thompson}
\address{Department of Mathematics and Statistics, Hunter College and the Graduate Center, CUNY, 
New York, NY 10065}
\email{robert.thompson@hunter.cuny.edu}

\numberwithin{equation}{section}
\newtheorem{theorem}[equation]{Theorem}
\newtheorem{definition}[equation]{Definition}
\newtheorem{proposition}[equation]{Proposition}
\newtheorem{lemma}[equation]{Lemma}

\newtheorem{corollary}[equation]{Corollary}

\DeclareMathOperator{\ext}{Ext}
\DeclareMathOperator{\Hom}{Hom}

\DeclareMathOperator{\End}{End}
\DeclareMathOperator{\coker}{coker}
\DeclareMathOperator{\spec}{Spec}
\DeclareMathOperator{\Gal}{Gal}
\DeclareMathOperator{\Map}{Map}
\DeclareMathOperator{\cotensor}{\Box}

\begin{document}

\begin{abstract}
The Bousfield-Kan (or unstable Adams) spectral sequence can be constructed for various homology theories such as Brown-Peterson homology theory BP, Johnson-Wilson theory $E(n)$, or Morava $E$-theory $E_n$. For nice spaces the $E_2$-term is given by Ext in a category of unstable comodules.    We establish an unstable Morava change of rings isomorphism between \break $\text{Ext}_{\mathcal{U}_{\Gamma_B}}(B,M)$ and  
$\text{Ext}_{\mathcal{U}_{E_{n*}E_n}/I_{n}}(E_{n*}/I_{n},E_{n*}\otimes_{BP_*} M)$
where $(B,\Gamma_B)$ denotes the Hopf Algebroid $(v_n^{-1}BP_*/I_{n}, v_n^{-1}BP_*BP/I_{n}   )$.   We show that the latter groups can be interpreted as $\text{Ext}$ in the category of continuous modules over the profinite monoid of endomorphisms of the Honda formal group law. By comparing this with the cohomology of the Morava stabilizer group we obtain an unstable Morava vanishing theorem when $p-1 \nmid n$.   
\end{abstract}

\maketitle

%%%%%%%%%%%%%%%%%%%%%%%%%%%%%%%%%%%%%%%%%%%
%
%  \section{Introduction}
%
%%%%%%%%%%%%%%%%%%%%%%%%%%%%%%%%%%%%%%%%%%%%
\section{Introduction}

In \cite{BCM} it is shown that the unstable Adams spectral sequence, as formulated by Bousfield and Kan \cite{BK}, can be used with a generalized homology theory represented by a $p$-local ring spectrum $E$ satisfying certain hypotheses, and for certain spaces $X$. The main example is $E=BP$. In these cases the effectiveness of the spectral sequence is demonstrated by: 1) setting up the spectral sequence and proving convergence, 2) formulating a general framework for computing the $E_2$-term, and 3)  computing the one and two line in the case where $E=BP$ and $X=S^{2n+1}$. 

In \cite{BT} the present author and M. Bendersky showed that this framework can be extended to periodic homology theories such as the Johnson-Wilson spectra $E(n)$.  However the approach to convergence in \cite{BT} is different from that in \cite{BCM}.  In the latter the Curtis convergence theorem is used to obtain a general convergence theorem based on the existence of a Thom map \[E\ \to H\mathbf{Z}_{(p)}\] and a tower over $X$.  This necessitates that $E$ be connective.  Obviously this doesn't apply to periodic theories such as $E(n)$. In \cite{BT} we study a tower under $X$ and define the $E$-completion of $X$ to be the homotopy inverse limit of this tower.   Convergence of the spectral sequence to the completion is guaranteed by, for example, a vanishing line in the $r$th term of the spectral sequence.  For the example of $E(1)$ and $X = S^{2n+1}$ we compute the $E_2$-term, for $p$ odd, and obtain such a vanishing line. 

It should be noted that the spectral sequence has been used to good effect in the work of Davis and Bendersky, in computing $v_1$-periodic homotopy groups of Lie Groups.  It should also be noted that the construction of an $E$-completion given in \cite{BT} has been strongly generalized by Bousfield in \cite{BO9}.     Also, the framework for the construction of the spectral sequence and the computation of the $E_2$-term in  \cite{BCM} and 
\cite{BT} has been generalized by Bendersky and Hunton in \cite{BH} to the case of an arbitrary Landweber Exact ring spectrum $E$.  This includes complete theories such as Morava $E$-theory.

In \cite{BH} the authors define an $E$-completion of $X$, and a corresponding Bousfield-Kan spectral sequence, for any space $X$ and any ring spectrum $E$, generalizing the construction of \cite{BT}.   If one further supposes that $E$ is a Landweber exact spectrum then the authors show that one can define a category of unstable comodules over the Hopf algebroid $(E_*, E_*(E))$.   This is accomplished by studying the primitives and indecomposables in the Hopf ring of $E$, extending the work of \cite{BCM}, \cite{BE5}.   Letting $\mathcal{U}$ denote this category of unstable comodules they show, for example, that if $X$ is a space such that $E_*(X) \cong \Lambda(M)$, an exterior algebra on the $E_*$-module $M$ of primitives, where $M$ is a free $E_*$-module concentrated in odd degrees, then the $E_2$-term of the spectral sequence can be identified as 
\[E_2^{s,t}(X) \cong \ext^s_{\mathcal{U}}(E_*(S^t),M).\]
This is Theorem 4.1 of \cite{BH}.  In the literature $E_{*}(S^t)$ is often abbreviated to $E_{*}[t]$ and this bigraded Ext group is denoted by the shorthand
\[  \ext^{s}_{\mathcal{U}}(E_{*}[t],M) \,\, \text{or even} \,\, \ext^{s,t}_{\mathcal{U}}(M).\]

There remains the problem of convergence and the problem of computing the $E_2$-term.    In this paper, following work on the case of Morava $K$-theory in D. Mulcahey's thesis \cite{MUL}, we extend the definition of an unstable comodule to certain torsion Hopf algebroids, and establish bounds on the cohomological dimension of the unstable Ext groups.
This involves an unstable version of the Morava change of rings theorem going from $v_{n}^{-1}BP/I_{n}$ to Morava $K$-theory, and then identification of the unstable cohomology as Ext groups in the category of continuous modules over $\End_n$, the profinite monoid of endomorphisms of $\Gamma_n$, the Honda formal group law, over $\mathbf{F}_{p^n}$.  The multiplication in $\End_n$ is given by composition.   The group of invertible endomorphisms is the well known Morava stabilizer group, and Morava theory tells us that the continuous cohomology of this group yields stable input into the chromatic machinery of stable homotopy theory.   Unstable information is obtained by considering non-invertible endomorphisms of $\Gamma_n$ as well. 

In the following theorem, the group on the left is Ext in the category of unstable comodules over the Hopf algebroid 
\[(B,\Gamma_{B}) = (B(n)_*,\Gamma_{B(n)_*}) = (v_{n}^{-1}BP_{*}/I_{n}, v_{n}^{-1}BP_*BP/I_{n}),\] 
and the group on the right is continuous Ext over the monoid 
$\End_n$, where $\Gal$ denotes the Galois group $\Gal(\mathbf{F}_{p^n}/\mathbf{F}_{p})$ and $E_{n*}$ is the coefficient ring of Morava $E$-theory.  We will denote  the unstable comodule which is the homology of the sphere by
\[B[k] = B(n)_*(S^{k}) = v_n^{-1}BP_*(S^{k})/I_{n} .\]
Let $M$ be an unstable $\Gamma_{B}$-comodule concentrated in odd dimensions.

%%%%%%%%%%%%%%%%%%%%%%%%%%%%%%%%
%% First main theorem
%%%%%%%%%%%%%%%%%%%%%%%%%%%%%%%

\begin{theorem}\label{first-main-theorem} 
There is an isomorphism
\[ \ext^{s}_{\mathcal{U}_{    \Gamma_{B}   }}(B[t],M)   \cong \ext^{s}_{\End_n}(     (E_{n})_{1}[t] /I_{n}      , (E_{n*}\otimes_{BP_*} M)_{1})^{\Gal}.
\]
\end{theorem}

In Section \ref{cohomological-dimension} we establish a relationship between the Ext groups over $\End_n$ and the cohomology of $S_n$, the Morava stabilizer group.  Using the cohomological dimension of $S_n$ (see \cite{RA2}) we obtain an unstable Morava vanishing theorem.

%%%%%%%%%%%%%%%%%%%%%%%%%%
%%   Second Main Theorem 
%%%%%%%%%%%%%%%%%%%%%%%%%%%%

\begin{theorem}\label{second-main-theorem}
Let $\Gamma_B$ and $M$ be as in Theorem \ref{first-main-theorem}.   Suppose $p-1 \nmid n$.  Then 
\[ \ext^{s}_{\mathcal{U}_{\Gamma_{B}}}(M[t],M)   = 0\quad \text{for}\quad  s > n^2+1\]
\end{theorem}

\subsection{Acknowledgements}

This work came out of an extended discussion with Martin Bendersky about unstable chromatic homotopy theory and the author is very grateful for all the insight he has provided.  In particular he was very helpful with Propositions \ref{reconciliation}, \ref{exactness-of-U-Landweber-exact-case}, and \ref{exactness-of-U-non-Landweber-exact-case}.  The author also wishes to thank Mark Hovey for help in understanding the nature of faithfully flat extensions of Hopf algebroids and equivalences of categories of comodules.    Also,  thanks are due to Hal Sadofsky and Ethan Devinatz for several useful conversations.   The author especially wishes to thank the referee for multiple careful readings of the paper and for pointing out a number of typos and flaws in the exposition, and for identifying several significant gaps which led to some erroneous claims in the preliminary versions of the paper.

%%%%%%%%%%%%%%%%%%%%%%%%%%%%%%%%%%%%%%%%%%%
%
%  \section{Unstable comodules}
%
%%%%%%%%%%%%%%%%%%%%%%%%%%%%%%%%%%%%%%%%%%%%

\section{Unstable Comodules}\label{section-unstable-comodules}

We begin by recalling some notions from \cite{BCM}, \cite{BH} and \cite{BT}.  Suppose that $E$ is a spectrum representing a Landweber exact homology theory with coefficient ring concentrated in even degrees. Let 
$\underline{E}_*$ denote the corresponding $\Omega$-spectrum.  There are generators $\beta_i \in E_{2i}(CP^{\infty})$ and under the complex orientation for complex cobordism $CP^{\infty} \to \underline{MU}_2$ these map to classes $E_{2i}(\underline{MU}_2)$.  Localized at a prime $p$, denote the image of $\beta_{p^i}$ by $b_{(i)} \in E_{2p^i}(\underline{E}_2)$.  Let $b_i \in E_{2p^i-2}(E)$ denote the image under stabilization.  Following \cite{BCM}, \cite{BH}, and \cite{BOAR},  when $E= BP$, we replace the elements $b_i$ with $h_i = c(t_i)$ and replace $b_{(i)}$ with $h_{(i)}$, a canonical lift of $h_i$.    For a finite sequence of integers $J=(j_1,j_2,\dots,j_n)$ define the {\em length} of $J$ to be $l(J) = j_1 +j_2 + \cdots j_n$ and define 
\[ h^J = h_1^{j_1}h_{2}^{j_2}\cdots h_n^{j_n}.\]

\begin{definition}\label{original-definition}  Let $(A,\Gamma)$ denote the Hopf algebroid $(E_*,E_*E)$ for a Landweber exact spectrum $E$.  Let $M$ be a free, graded $A-module$.  Define $U_{\Gamma}(M)$ to be sub-$A$-module of $\Gamma\otimes_{A}M$ spanned by all elements of the form $h^J\otimes m$ where $2l(J) <  |m|$.    Secondly, define $V_{\Gamma}(M)$ to be sub-$A$-module of $\Gamma\otimes_{A}M$ spanned by all elements of the form $h^J\otimes m$ where $2l(J) \le  |m|$.
\end{definition}
We will sometimes drop the subscript $\Gamma$ from the notation if it will not cause confusion.

The following theorem was proved in \cite{BCM} for $E=BP$ and in \cite{BH} for an arbitrary Landweber exact theory.   Here $M_s$ denotes a free $A$-module generated by one class $i_s$ in dimension $s$.

\begin{theorem}\label{suspension-isomorphism}
In the Hopf ring for $E$ the suspension homomorphism restricted to the primitives 
\[\sigma_*: PE_*(\underline{E}_s) \to U(M_s)\]
and the suspension homomorphism restricted to the indecomposables
\[\sigma_*: QE_*(\underline{E}_s) \to V(M_s)\]
are isomorphisms.  
\end{theorem}

The functors $U_{\Gamma}(-)$ and $V_{\Gamma}(-)$  extend to the category of arbitrary $A$-modules. 
\begin{definition}\label{extend-U}
 Let $M$ be an $A$-module and   
let \[F_1 \to F_0  \to M \to 0\]
be exact with $F_1$ and $F_0$ free over $A$.  Define $U_{\Gamma}(M)$ by 
\[U_{\Gamma}(M) = \coker(U_{\Gamma}(F_1) \to U_{\Gamma}(F_0))\]
and $V_{\Gamma}(M)$ by 
\[V_{\Gamma}(M) = \coker(V_{\Gamma}(F_1) \to V_{\Gamma}(F_0)).\]
\end{definition}

It is shown in \cite{BCM}, \cite{BH} that $U$ and $V$ are each the functor of a comonad $(U,\Delta,\epsilon)$ and $(V,\Delta,\epsilon)$ on the category of $A$-modules.  For now we will focus on the functor $U$ but in everything that follows in this section and the next there are analogous results for $V$.  Keep in mind that if $M$ is concentrated in odd dimensions, then $U(M)$ and $V(M)$ are the same.

Using some work from Dustin Mulcahey's thesis \cite{MUL} we can extend the above to a more general situation.  But first we will reconcile two differently defined but apparently similar notions of $U(M)$.    We still suppose $E$ is a spectrum representing a Landweber exact homology theory and let $M$ denote a free $E_*$-module, and let $F$ be any $p$-local homology theory which is torsion free with coefficients concentrated in even dimensions.   In Definition 2.9 of \cite{BH}, Bendersky and Hunton define $U_F(M)$ to be the sub-$F_*$-module of $F_*(E)\otimes_{E_*}M$ spanned by elements $h^{I}\otimes m$ where $2l(I) < |m|$.    If we let $F= BP$ this gives a $BP_*$-module $U_{BP}(M)$.  However, regarding $M$ as a $BP_*$-module, which will not typically be a free $BP_{*}$-module,  we also have the $BP_*$-module $U_{\Gamma}(M)$, where $BP_*BP = \Gamma$,  defined in \ref{extend-U} above.  Note that $U_{BP}(M)$ is a $BP_*$-submodule of $BP_*(E)\otimes_{E_*}M = BP_*BP\otimes_{BP_*}E_*\otimes_{E_*}M = BP_*BP\otimes_{BP_*}M$ by definition, whereas $U_{\Gamma}(M)$ maps to $BP_*BP\otimes_{BP_*}M$, but not {\em a priori} injectively.  

\begin{proposition}\label{reconciliation}  Denote $(BP_*,BP_*BP)$ by $(A,\Gamma)$ and $(E_*,E_*E)$ by $(B,\Sigma)$ for a Landweber exact homology theory $E$.  Let $M$ denote a free $E_*$-module. Then the map 
$U_{\Gamma}(M) \to \Gamma \otimes_{A}M$ which comes from Definitions \ref{original-definition} and \ref{extend-U} is an injection and $U_{BP}(M) \cong U_{\Gamma}(M)$.    Furthermore $B\otimes_{A} U_{\Gamma}(M) \cong   
U_{\Sigma}(M)$.  A similar result holds for $V$. 
\end{proposition}

\begin{proof}  Because $E$ is Landweber exact, it is torsion free, and it follows by a simple argument that $U_{\Gamma}(M)$ is also torsion free.  So to establish injectivity it suffices to tensor with the rationals: \break $Q\otimes U_{\Gamma}(M) \to Q\otimes \Gamma \otimes_{A}M$.  This map is in fact an isomorphism because rationally the unstable condition is vacuous. The isomorphism 
$U_{BP}(M) \cong U_{\Gamma}(M)$ follows by comparing the two images in $BP_*BP\otimes_{BP_*}M$.  The last statement follows immediately from Corollary 2.12 of \cite{BH}.
\end{proof}
It follows that $B\otimes_{A} U_{\Gamma}(M) \cong   U_{\Sigma}(M)$ for an arbitrary $B$-module, and similarly for $V$.

\begin{definition}\label{unstable-comodule-categories-Landweber-exact}
 Let $\mathcal{U}_{\Sigma}$ denote the category of coalgebras over the comonad $U_{\Sigma}$ and similarly let $\mathcal{V}_{\Sigma}$ denote the category of coalgebras over the comonad $V_{\Sigma}$.  We call an object in $\mathcal{U}_{\Sigma}$ (or in $\mathcal{V}_{\Sigma}$, depending on the context) an unstable $\Sigma$-comodule. 
\end{definition}

A Hopf algebroid $(B,\Sigma)$ is called {\em flat} if $\Sigma$ is flat as a left (and hence right) $B$-module.  Flatness ensures that the category of 
$\Sigma$-comodules is an abelian category.   We also want the category of unstable comodules to be abelian. This will follow from the exactness of the functor $U_{\Sigma}$  on the category of $B$-modules.   The exactness of $U$ for $(BP_*,BP_*BP)$ has been asserted without proof in multiple places in the literature.   A proof is given in \cite{BCR} but that proof only applies to the case of free modules. 

The following proof is based on an idea of Martin Bendersky's and we are grateful to him for allowing it to be included here. In particular Bendersky suggested using the Boardman basis for the Hopf ring for $BP$ which is the most convenient for this purpose. 

\begin{proposition}\label{exactness-of-U-Landweber-exact-case}
Let $(A, \Gamma) = (BP_*,BP_*BP)$ and suppose $(B,\Sigma)$ is a Hopf algebroid associated to a Landweber exact homology theory.  Then the functors $U_{\Sigma}$ and $V_{\Sigma}$ are exact on $B$-modules. 
\end{proposition}

\begin{proof} 

We first give the proof for $BP$.  We need to use the fact that the indecomposables and the primitives in the Hopf ring for $BP$ are free $BP_*$-modules.   In \cite{RW1} Ravenel and Wilson compute $BP_*(\underline{BP}_*)$, the Hopf ring for $BP$, and show that the indecomposables and primitives are free {\em left} $BP_*$-modules.   They write a down a basis, which in turn gives a basis as a free $Z_{(p)}$-module.   In spite of the fact that there is no conjugation in the Hopf ring corresponding to the conjugation $c$ in $BP_*BP$, Boardman proves in \cite{BOAR} that if one considers the right action of $BP_*$ instead, the indecomposables and primitives are free {\em right} $BP_*$-modules as well. 

To start we restrict to the even spaces in the Hopf ring and the functor $V$.   The proof can then be extended to the odd spaces and the functor $U$ by standard arguments. By \cite{RW1} the indecomposables 
$QBP_*(\underline{BP}_*)$ are generated as a left $BP_*$-module by monomials $h^{\circ J'}\circ[v^K]$, where $J' = (j_0,j_1,\dots)$ and $K =(k_1,k_2,\dots)$ are sequences of non-negative integers.  If also $I=(i_1,i_2,\dots)$, the bidegree of a generator is given by 
\[ v^Ih^{J'}[v^K] \in QBP_{|v^I| + |h^{J'}|+2l(J')}(\underline{BP}_{2l(J')-|v^K|})\] 
where $v^Ih^{J'}[v^K]$ stands for $v^{I}\circ h_0^{\circ j_0}\circ h_1^{\circ j_1} \circ \dots \circ [v^K]$. 
The isomorphism in \ref{suspension-isomorphism} is given by 
\[ v^Ih^{J'}[v^K]  \xrightarrow{}  v^Ih^{J}\otimes v^K\iota_{2m}\]
where $J = (j_1,j_2,\dots)$ is obtained from $J'$ by dropping $j_0$ and $2m = 2j_0 + 2l(J) - |v^K|$.   

The Ravenel-Wilson basis involves a condition on $J'$ and $K$.   
For the Boardman basis consider monomials $v^Ih^{J'}$ and call such a monomial {\em Boardman allowable} if it is not divisible by any monomial of the form 
\[v_{d_0}v_{d_1}^{p}v_{d_2}^{p^2}\dots v_{d_l}^{p^l}h_{l}\]
where $l \ge 0$ and $d_0 \le d_1 \le d_2 \le \dots \le d_l$.   
Then the main theorem of \cite{BOAR} is that the Boardman allowable monomials are a basis for the indecomposables as a right $BP_*$-module.   (Note that while $QBP_*(\underline{BP_{2m}})$ is a left $BP_*$-submodule it is not a right $BP_*$-submodule because multiplying by an element $v\in BP_*$ on the right changes the index of the space in the $\Omega$-spectrum. So we have to consider the entire Hopf ring when we consider the right module structure.) 

If we let $\mathcal{B}$ denote the free $Z_{(p)}$-module generated by the Boardman allowable monomials then Boardman's theorem implies that
\[QBP_*(\underline{BP}_*) \cong \mathcal{B} \otimes_{Z_{(p)}} A\]
as free $Z_{(p)}$-modules.   (The bigrading on the Hopf ring is not the tensor product of gradings on the two factors.) 

Now suppose $F$ is free graded $BP_*$-module of rank one on a generator in dimension $2m$.  Then by \ref{suspension-isomorphism} we identify $V(F)$ with the subgroup of $\mathcal{B} \otimes_{Z_{(p)}} A$ spanned by the monomials with second bidegree equal to $2m$.  If $F$ is a free graded $BP_*$-module of arbitrary rank we identify $V(F)$ with a subgroup of a sum of copies of $\mathcal{B} \otimes_{Z_{(p)}} A $ by identifying each summand of $V(F)$ with a subgroup of a copy of $\mathcal{B} \otimes_{Z_{(p)}} A $.  If $f:F_1\to F_0$ is a map of free $BP_*$-modules then this defines an evident map $\mathcal{B}\otimes f$ from a sum of copies of $\mathcal{B} \otimes_{Z_{(p)}} A$ to another sum of copies of $\mathcal{B} \otimes_{Z_{(p)}} A$, which restricts to $V(f)$.   It follows that for any $BP_*$-module $M$, with a free presentation
\[F_1 \to F_0 \to M \to 0\] 
the $Z_{(p)}$-module $V(M)$ is isomorphic to the subgroup of $\mathcal{B}\otimes_{Z_{(p)}} M$ of elements in the appropriate bidegree.   It is immediate from this that $V$ is an exact functor.

For the arbitrary Landweber exact case suppose  we have short exact sequence of $B$-modules
\[ 0 \to M' \to M \to M'' \to 0.\] Thinking of this as a SES of $A$-modules we have a SES 
\[ 0 \to U_{\Gamma}(M') \to U_{\Gamma}(M) \to U_{\Gamma}(M'') \to 0\] 
since $U_{\Gamma}$ is exact.     These are unstable $\Gamma$-comodules, hence stable $\Gamma$-comodules. Since $B$ is Landweber exact tensoring with $B$ preserves exactness on the category of $\Gamma$-comodules, and the result follows from the sentence that follows Proposition \ref{reconciliation}.

\end{proof}

Now, following Mulcahey (\cite{MUL}) we can generalize the definition of unstable comodules to certain non-Landweber exact homology theories.  For the time being $(A, \Gamma)$ still denotes $(BP_*,BP_*BP)$.
Suppose $A\xrightarrow{f} B$ is a map of graded algebras.  If we define 
\[{\Sigma = B \otimes_{A} \Gamma \otimes_{A} B}\] then $(B,\Sigma)$ becomes a Hopf algebroid and we have a map of Hopf algebroids $(A,\Gamma) \to (B, \Sigma)$.    The example that was treated in \cite{MUL} was $A=BP_*$ and $B = K(n)_*$ but the following makes sense more generally.

\begin{definition}\label{mulcahey-definition}  The endofunctor $U_{\Sigma}$ on $B$-mod, the category of $B$-modules, is defined  by 
\[U_{\Sigma}(N) = B \otimes_A U_{\Gamma}(N).\]

Define a comultiplication by 
\[
\begin{diagram}
\node{   U_{\Sigma}(N) = B\otimes_A U_{\Gamma}(N)   }    \arrow{e,t}{  B\otimes \Delta^{\Gamma}   }  \arrow{se,t}{\Delta^{\Sigma}}  \node{   B\otimes_A U^{2}_{\Gamma}(N)   }
\arrow{s,r}{B\otimes U_{\Gamma}(f\otimes U_{\Gamma}(N))}    \\
\node{}    \node{   B\otimes_A U_{\Gamma}(B\otimes_A U_{\Gamma}(N))    }\\
\end{diagram}
\]
and a counit
\[
U_{\Sigma}(N) = B\otimes_A U_{\Gamma}(N)  \xrightarrow{B\otimes \epsilon^{\Gamma}} B\otimes_A N \xrightarrow{} N
\]
Make an analogous definition for $V_{\Sigma}$. 
\end{definition}

\begin{proposition}[See \cite{MUL}]  The functors $U_{\Sigma}$ and $V_{\Sigma}$  are both comonads on the category of $B$-modules. \end{proposition}
\begin{proof} The proof is a straightforward diagram chase.\end{proof}

By Proposition \ref{reconciliation} this generalizes the definition of $U$ and $V$ in the Landweber exact case.

\begin{definition}\label{unstable-comodule-categories}
Still denoting $(A, \Gamma) = (BP_*,BP_*BP)$, let $\mathcal{H}$ denote the category whose objects are Hopf algebroids $(B,\Sigma)$ 
arising from a map of commutative  graded algebras $A\to B$, with $\Sigma = B\otimes_{A}\Gamma \otimes_{A}B$ as above, and
satisfying
\begin{enumerate}
\item $\Sigma$ is flat as a $B$-module
\item $U_{\Sigma}$ is an exact functor
\end{enumerate} 
The morphisms in $\mathcal{H}$ are Hopf algebroid maps $(B,\Sigma) \xrightarrow{j} (C,\Phi)$ under $(A,\Gamma)$, i.e. a commutative diagram of Hopf algebroids.
\[
\begin{tikzcd}
& (B,\Sigma) \arrow[dd,"j"] \\
(A,\Gamma)    \arrow[ur]  \arrow[dr] & \\
 & (C,\Phi) \\
\end{tikzcd}
\]
Define $\mathcal{U}_{\Sigma}$ and $\mathcal{V}_{\Sigma}$ just as in Definition \ref{unstable-comodule-categories-Landweber-exact}. 
 \end{definition}

Thus an unstable $\Sigma$-comodule has a lifting:

\[
\begin{diagram}
\node{M} \arrow{e,t}{} \arrow{se,b}{\psi_{M}} \node{\Sigma\otimes_{B}M} \\
\node{} \node{U_{\Sigma}(M)} \arrow{n}\\
\end{diagram}
\]

Now we will let $(A,\Gamma)$ denote an arbitrary object in $\mathcal{H}$.  %The category $\mathcal{U}_{\Gamma}$ is an abelian category and 
The functor $U_{\Gamma}$ restricted to $\mathcal{U}_{\Gamma}$ is the functor of a monad $(U_{\Gamma},\mu,\eta)$, using the definitions $\mu = U_{\Gamma}\epsilon$ and $\eta=\psi$.  The Ext groups in $\mathcal{U}_{\Gamma}$ are defined and computed as follows (see \cite{BCM}, \cite{BH} and \cite{BT}). 

\begin{definition}   Suppose $M$ is an unstable comodule.
Analogous to the stable case, the monad $(U,\mu,\eta)$ gives maps
\begin{align*}
 & U^{i}\eta^{U}U^{n-i}: U^{n}(M)\to U^{n+1}(M),\,\, 0\le i \le n,\\
 & U^{i}\mu^{U}U^{n-i}: U^{n+2}(M)\to U^{n+1}(M),\,\, 0\le i \le n,
\end{align*}
which define a cosimplicial object in $\mathcal{U}$ called the cobar resolution.   For each $t \ge 0$ let  $A[t]$ denote a free $A$-module of rank one on a generator with dimension $t$.  Apply the functor 
$\hom_{\mathcal{U}}(A[t],\,\,)$ to get a cosimplicial abelian group and hence a
chain complex called the cobar complex
\[\label{cobar-complex}
\hom_{\mathcal{U}}(A[t],U(M)) \xrightarrow{\partial} \hom_{\mathcal{U}}(A[t],U^{2}(M)) \xrightarrow{\partial} 
\hom_{\mathcal{U}}(A[t],U^3(M))\xrightarrow{\partial} \cdots\\
\]
with
\[\partial = \sum_{i=0}^{n} (-1)^i d^{i}: \hom_{\mathcal{U}}(A[t],U^{n}(M)) \to 
\hom_{\mathcal{U}}(A[t],U^{n+1}(M)).\]
Here $d^{i} = \hom_{\mathcal{U}}(A[t],U^{i}\eta^{U}U^{n-i})$.  By the adjunction 
\[\hom_{\mathcal{U}}(A[t],U(N)) = \hom_{A-\text{mod}}(A[t],N) = N_t\]
the cobar complex becomes
\[ M_t \xrightarrow{\partial} U(M)_t \xrightarrow{\partial} U^2(M)_t \xrightarrow{\partial} \dots   \]
The homology of this chain complex gives $\ext^{s,t}_{\mathcal{U}}(A,M)$. 
\end{definition}

In \cite{MR1} Miller and Ravenel consider a morphism of Hopf algebroids  $(A,\Gamma) \to (B,\Sigma)$ and define a pair of adjoint functors on the comodule categories

\begin{center}
\begin{tikzpicture}
\node (a) at (0,0) {$  \Gamma\text{-comod} $};
\node (b) at (3,0) {$  \Sigma\text{-comod}  $};
\path[->,font=\scriptsize,>=angle 90]
([yshift= 3pt]a.east) edge node[above] {$\pi_*$} ([yshift= 3pt]b.west);
\path[<-,font=\scriptsize,>=angle 90]
([yshift= -3pt]a.east) edge node[below] {$\pi^*$} ([yshift= -3pt]b.west);
\end{tikzpicture}
\end{center}

\noindent defined by $\pi_*(M) = B \otimes_A M$ and $\pi^*(N) = (\Gamma \otimes_A B)\cotensor_{\Sigma} N$ for a $\Gamma$-comodule $M$ and a $\Sigma$-comodule $N$.
This adjunction is discussed in detail in several places, for example \cite{HOV2} and \cite{MUL}.  The functors $\pi_*$ and $\pi^*$ often define inverse equivalences of comodule categories.
For example if $\Sigma = B\otimes_{A} \Gamma \otimes_{A} B$ and $A\to B$ is a faithfully flat extension of rings, then it is not difficult to see that this is the case.

Now suppose that $(A,\Gamma)\to (B,\Sigma)$ is a morphism in $\mathcal{H}$.  Following Mulcahey's work in \cite{MUL} we define unstable analogs of $\pi_*$ and $\pi^*$.  

\begin{definition}\label{unstable-adjoint-functors}
Define functors 
\begin{center}
\begin{tikzpicture}
\node (a) at (0,0) {$ \mathcal{U}_{ \Gamma} $};
\node (b) at (3,0) {$ \mathcal{U}_{ \Sigma}  $};
\path[->,font=\scriptsize,>=angle 90]
([yshift= 3pt]a.east) edge node[above] {$\alpha_*$} ([yshift= 3pt]b.west);
\path[<-,font=\scriptsize,>=angle 90]
([yshift= -3pt]a.east) edge node[below] {$\alpha^*$} ([yshift= -3pt]b.west);
\end{tikzpicture}
\end{center}
\noindent by $\alpha_*(M) = B \otimes_A M$ for an unstable $\Gamma$-comodule $M$, and for an unstable $\Sigma$-comodule $N$, define $\alpha^*(N)$ to be the equalizer

\begin{center}
\begin{tikzpicture}
\node (a) at (0,0) {$\alpha^*(N)$};
\node (b) at (2,0) {$ U_{\Gamma}(N)  $};
\node (c) at (5,0) {$ U_{\Gamma}U_{\Sigma }(N)$};
\path[->,font=\scriptsize,>=angle 90]
(a) edge (b)
([yshift= 3pt]b.east) edge node[above] {$ U_{\Gamma}(\psi_{N}) $} ([yshift= 3pt]c.west)
([yshift= -3pt]b.east) edge node[below] {$U_{\Gamma}(\beta)\circ \Delta_{\Gamma}  $} ([yshift= -3pt]c.west);
\end{tikzpicture}
\end{center}
\end{definition}
\noindent where $\beta:U_{\Gamma}(N)\to U_{\Sigma}(N).$

\begin{proposition}[See \cite{MUL}]  The functors $\alpha_*$ and $\alpha^{*}$  form an adjoint pair. \end{proposition}
\begin{proof} This follows by considering the map 
\[ B\otimes_A U_{\Gamma}(M) \xrightarrow{} B\otimes_A U_{\Gamma}(B\otimes_A M)\]
which is natural in the $A$-module $M$ and gives a morphism of comonads $U_{\Gamma}\to U_{\Sigma}$ which leads to the adjoint pair on comodule categories. See \cite{MUL} for the details.\end{proof}

As an example we describe some torsion unstable $BP_*BP$-comodules.  These will not be used in this paper, but are included to illustrate the unstable condition of Definition \ref{original-definition}. Stably, for every $n$,  $BP_*/I_{n}$ is a $BP_*BP$-comodule and $v_n$ is a comodule map mod $I_n$.  This is because $I_n$ is an invariant ideal.  Unstably there is a subtlety because the terms in $\eta_{R}(v_k) - v_k$ may not lie in $U(BP_*(S^m)/I_{n})$ if  the dimension of the sphere is too small.  For example, $\eta_{R}(v_1) = v_1 - ph_1$, however $\eta_R(v_1) - v_1$ isn't divisible by $p$ in $U(BP_*(S^1))$ since $h_1$ doesn't live on the circle.  You need to be on the 3-sphere or higher for $v_1$ to be an unstable comodule map, which makes  $BP_*(S^m)/I_2$ into an an unstable comodule.

\begin{proposition}\label{benderskys-observation} Given $n$ and $p$, $BP_*(S^m)/I_{n}$ is an unstable comodule and 
\[BP_*(S^m)/I_{n}   \xrightarrow{v_n} BP_*(S^m)/I_{n}\]
is an unstable comodule map, as long as $m \ge 2(\dfrac{p^n-1}{p-1}) +1$.
\end{proposition}

\begin{proof} The statement is true for $n=1$ by the example above.  Let $n\ge 1$ and assume $m$ is as stated. Inductively $BP_*(S^m)/I_{n}$
is an unstable comodule because it is the cokernel of multiplication by $v_{n-1}$ on $BP_*(S^m)/I_{n-1}$. Consider $\eta_{R}(v_n) - v_n$ which is a polynomial in the $h's$ and $\eta_{R}(v)'s$.  The largest length monomial in the $h's$ which could occur is $h_{1}^{\alpha}$ with $|h_{1}^{\alpha}| = |v_n|$, i.e. $2(p-1)\alpha = 2(p^n-1)$.  Therefore with $m \ge 2(\dfrac{p^n-1}{p-1}) +1$ we have that $\eta_{R}(v_n) - v_{n} =0$ in $U(BP_*(S^m)/I_{n})$.
\end{proof}

This result is not sharp. A stronger statement is possible but we will not pursue that here.

%%%%%%%%%%%%%%%%%%%%%%%%%%%%%%%%%%%%%%%%%%%
%
%  \section{Faithfully Flat Extensions}
%
%%%%%%%%%%%%%%%%%%%%%%%%%%%%%%%%%%%%%%%%%%%%

\section{Faithfully Flat Extensions}

The following theorem is an unstable version of a theorem due to Mike Hopkins, Mark Hovey, and Hal Sadofsky.  See   \cite{HOP}, \cite{HOV}, and \cite{HOVSA}. Hovey's paper \cite{HOV} has a detailed proof of the stable theorem in the form that we need, which is stated below as Theorem \ref{hovey's-theorem}.  The proof in \cite{HOV} is based on a study of the category of quasi-coherent sheaves on a groupoid scheme.  That theory has not yet been developed in an unstable setting but we don't need that for the present work. The author is very grateful to Mark Hovey for a detailed discussion of various aspects of Theorem \ref{hovey's-theorem} below. 

\begin{theorem}\label{unstable-faithflat-equiv}  
Suppose  $(A,\Gamma)\to (B,\Sigma)$ is a map of  Hopf algebroids in $\mathcal{H}$.  Assume there exists an algebra $C$ along with an algebra map $B \otimes_A \Gamma \xrightarrow{g} C$ such that the composite
\[A \xrightarrow{1 \otimes \eta_R} B \otimes_A \Gamma \xrightarrow{g} C\]
is a faithfully flat extension of $A$-modules.  To be explicit the first map is the one that takes $a$ to $1\otimes \eta_R(a)$.  
Then $\alpha_*$ and $\alpha^*$ of \ref{unstable-adjoint-functors} are adjoint inverse equivalences of categories.
\end{theorem}

The existence of the map $g$ satisfying the stated condition generalizes the condition of $A\to B$ being faithfully flat. 
 
The stable theorem on which this is based says
\begin{theorem}[Hopkins, Hovey, Hovey-Sadofsky]\label{hovey's-theorem}   Let $(A,\Gamma)\to (B,\Sigma)$ be a map of flat Hopf algebroids such that $\Sigma = B \otimes_{A} \Gamma \otimes_{A} B$, and assume there exists an algebra $C$ along with an algebra map $B \otimes_A \Gamma \xrightarrow{g} C$ such that the composite
\[A \xrightarrow{1 \otimes \eta_R} B \otimes_A \Gamma \xrightarrow{g} C\]
is a faithfully flat extension of $A$-modules.  Then  
\[\Gamma\text{-comod}  \xrightarrow{\pi_*} \Sigma\text{-comod}\]
is an equivalence of categories.   
\end{theorem}

This enables the following lemma.
\begin{lemma}\label{hovey's-lemma}  Let $(A,\Gamma) \to (B,\Sigma)$ satisfy the hypotheses of \ref{unstable-faithflat-equiv}.  Recall that $\alpha_*:\mathcal{U}_{\Gamma} \to \mathcal{U}_{\Sigma}$ is given by $\alpha_*(M) = B\otimes_{A}M$. 
Let $f:M\to N$ be a morphism in $\mathcal{U}_{\Gamma}$.  Then $\alpha_*(f)$ is an isomorphism if and only if $f$ is an isomorphism.
Furthermore $\alpha_*$ is exact.
 \end{lemma}

\begin{proof}
By Theorem \ref{hovey's-theorem} the functor $\pi_*$ is exact since an equivalence of abelian categories is an exact functor.  An unstable $\Gamma$-comodule map is a stable $\Gamma$-comodule map and a sequence in $\mathcal{U}_{\Gamma}$ is exact if and only if it's exact in $\Gamma\text{-comod}$, so $\alpha_*$ is exact on  $\mathcal{U}_{\Gamma}$.  A similar argument gives the first statement.
\end{proof}

\begin{proof}[Proof of Theorem \ref{unstable-faithflat-equiv}]  
For $N \in \mathcal{U}_{\Sigma}$ consider the counit of the adjunction
\[ \alpha_*\alpha^* N \xrightarrow{} N.\]
By Lemma \ref{hovey's-lemma} $\alpha_*\alpha^* N = B\otimes_{A} \alpha^*(N)$ sits in an equalizer diagram 
 
 \begin{center}
\begin{tikzpicture}
\node (a) at (0,0) {$B\otimes_{A} \alpha^*(N)$};
\node (b) at (3,0) {$B\otimes_{A} U_{\Gamma}(N)  $};
\node (c) at (7.5,0) {$B\otimes_{A}  U_{\Gamma}U_{\Sigma }(N)$};
\path[->,font=\scriptsize,>=angle 90]
(a) edge (b)
([yshift= 3pt]b.east) edge node[above] {$B\otimes_{A} U_{\Gamma}(\psi_{N}) $} ([yshift= 3pt]c.west)
([yshift= -3pt]b.east) edge node[below] {$B\otimes_{A} U_{\Gamma}(\beta)\circ \Delta_{\Gamma}  $} ([yshift= -3pt]c.west);
\end{tikzpicture}
\end{center}
\noindent which is the same thing as 
 \begin{center}
\begin{tikzpicture}
\node (a) at (0,0) {$B\otimes_{A} \alpha^*(N)$};
\node (b) at (2.5,0) {$ U_{\Sigma}(N)  $};
\node (c) at (6,0) {$U_{\Sigma}U_{\Sigma }(N)$};
\path[->,font=\scriptsize,>=angle 90]
(a) edge (b)
([yshift= 3pt]b.east) edge node[above] {$ U_{\Sigma}(\psi_{N}) $} ([yshift= 3pt]c.west)
([yshift= -3pt]b.east) edge node[below] {$  \Delta_{\Sigma}  $} ([yshift= -3pt]c.west);
\end{tikzpicture}
\end{center}
\noindent because $\Sigma = B\otimes_{A} \Gamma \otimes_{A} B$.  It follows that $B\otimes_{A}\alpha^*N \cong N$.  

For $M \in \mathcal{U}_{\Gamma}$ look at the unit of the adjunction
\[M \xrightarrow{} \alpha^*\alpha_* M.\]
The target sits in an equalizer diagram
 \begin{center}
\begin{tikzpicture}
\node (a) at (0,0) {$ \alpha^*\alpha_* M$};
\node (b) at (2.5,0) {$ U_{\Gamma}(B\otimes_{A} M)  $};
\node (c) at (6,0) {$U_{\Gamma}U_{\Sigma }(B\otimes_{A} M)$};
\path[->,font=\scriptsize,>=angle 90]
(a) edge (b)
([yshift= 3pt]b.east) edge node[above] {$  $} ([yshift= 3pt]c.west)
([yshift= -3pt]b.east) edge node[below] {$  $} ([yshift= -3pt]c.west);
\end{tikzpicture}
\end{center}
\noindent Tensor this with $B$
 \begin{center}
\begin{tikzpicture}
\node (a) at (0,0) {$B\otimes_{A} \alpha^*\alpha_* M$};
\node (b) at (3.5,0) {$B\otimes_{A} U_{\Gamma}(B\otimes_{A} M)  $};
\node (c) at (8,0) {$B\otimes_{A}U_{\Gamma}U_{\Sigma }(B\otimes_{A} M)$};
\path[->,font=\scriptsize,>=angle 90]
(a) edge (b)
([yshift= 3pt]b.east) edge node[above] {$  $} ([yshift= 3pt]c.west)
([yshift= -3pt]b.east) edge node[below] {$  $} ([yshift= -3pt]c.west);
\end{tikzpicture}
\end{center}
\noindent which gives
 \begin{center}
\begin{tikzpicture}
\node (a) at (0,0) {$B\otimes_{A} \alpha^*\alpha_* M$};
\node (b) at (3,0) {$ U_{\Sigma}(B\otimes_{A} M)  $};
\node (c) at (7,0) {$U_{\Sigma}U_{\Sigma }(B\otimes_{A} M)$};
\path[->,font=\scriptsize,>=angle 90]
(a) edge (b)
([yshift= 3pt]b.east) edge node[above] {$  $} ([yshift= 3pt]c.west)
([yshift= -3pt]b.east) edge node[below] {$  $} ([yshift= -3pt]c.west);
\end{tikzpicture}
\end{center}
\noindent So $B\otimes_{A} \alpha^*\alpha_*M \cong B\otimes_{A} M$.  The unit of the adjunction is an unstable $\Gamma$-comodule map so Lemma \ref{hovey's-lemma} applies and we have ${M  \xrightarrow{\cong}  \alpha^*\alpha_*M}$. 
\end{proof}

This equivalence of categories induces a change of rings isomorphism of Ext groups.  To be explicit we have
\begin{theorem}\label{unstable-faithflat-cor}
 Assume the hypotheses of \ref{unstable-faithflat-equiv}.
Then for any unstable $\Gamma$-comodule $M$, there is an isomorphism
\[ \ext^{s}_{\mathcal{U}_{\Gamma}}(A[t],M) \to   \ext^{s}_{\mathcal{U}_{\Sigma}}(B[t],B\otimes_{A} M).\]
\end{theorem}

First we make an observation.

\begin{lemma}\label{unstable-cotensor-lemma}
For an unstable $\Sigma$-comodule $N$  we have $\alpha^*U_{\Sigma}(N) = U_{\Gamma}(N)$.
\end{lemma}

\begin{proof}  
We have 
\begin{equation}\label{first-map} 
 U_{\Gamma}(N) \to \alpha^*U_{\Sigma}(N).
 \end{equation}  
  Tensor with $B$ to get 
\[B\otimes_{A} U_{\Gamma}(N) \to B\otimes_{A} \alpha^*U_{\Sigma}(N)\]
which is 
\[U_{\Sigma}(N) \xrightarrow{\cong} U_{\Sigma}(N).\]
By Lemma \ref{hovey's-lemma} the map \ref{first-map} is an isomorphism.
\end{proof}

\begin{proof}[Proof of \ref{unstable-faithflat-cor}]
Let 
\[N \to U_{\Sigma}(N) \to U_{\Sigma}^{2}(N) \to U_{\Sigma}^{3}(N) \to \dots \]
be the unstable cobar resolution for $N$.  Apply $\alpha^*$ to get
\begin{equation}\label{alpha-unstable-cobar}
\alpha^{*}N \to \alpha^{*}U_{\Sigma}(N) \to  \alpha^{*}U_{\Sigma}^{2}(N) \to  \alpha^{*}U_{\Sigma}^{3}(N) \to \dots
\end{equation}
Since $\alpha^*$ is an equivalence of abelian categories it is exact so \ref{alpha-unstable-cobar} is exact.

By Lemma \ref{unstable-cotensor-lemma} $\alpha^{*}U_{\Sigma}(N) = U_{\Gamma}(N)$ so \ref{alpha-unstable-cobar} is a resolution of $\alpha^*N$ by models in the category of unstable $\Gamma$-comodules
\[\alpha^{*}N \to U_{\Gamma}(N) \to  U_{\Gamma}U_{\Sigma}^{}(N) \to U_{\Gamma}U_{\Sigma}^{2}(N) \to \dots ,\]
and hence can be used to compute $\ext$ (see for example Theorem 2.3 of \cite{BCR}).   
Apply $\Hom_{\mathcal{U}_{\Gamma}}(A,\_)$ to get 
\[N \to U_{\Sigma}(N) \to U_{\Sigma}^{2}(N) \dots\]
which is the $\Sigma$-cobar complex for $N$.
This shows that 
\[ \ext_{\mathcal{U}_{\Gamma}}(A,\alpha^{*}N) \to   \ext_{\mathcal{U}_{\Sigma}}(B,N)\]
is an isomorphism.
Apply this to the case $N=\alpha_*M$ to get the result.
\end{proof}

%%%%%%%%%%%%%%%%%%%%%%%%%%%%%%%%%%%%%%%%%%%
%
%  \section{Morava $E$-theory}
%
%%%%%%%%%%%%%%%%%%%%%%%%%%%%%%%%%%%%%%%%%%%%
\section{Morava $E$-theory}\label{Morava-$E$-theory}

This section is based on the work of Morava \cite{MOR}. We will closely follow the exposition of Devinatz \cite{DEV}.  Let $W\mathbf{F}_{p^n}$ denote the Witt ring over $\mathbf{F}_{p^n}$,  the complete local $p$-ring having $\mathbf{F}_{p^n}$ as its residue field.  Let $\sigma$ denote the generator of the Galois group $\text{Gal} = \text{Gal}(\mathbf{F}_{p^n}/ \mathbf{F}_{p})$ which is cyclic of order $n$.  Note that $\text{Gal}$ acts on $W\mathbf{F}_{p^n}$ by 
\[(\sum_{i} w_i p^i)^{\sigma} = \sum_{i} w_i^p p^i\]
where the coefficients $w_i$ are multiplicative representatives.

Let $\Gamma_n$ be the height $n$ Honda formal group law over a field $k$ of characteristic $p$.   The endomorphism ring of $\Gamma_n$ over $k=\mathbf{F}_{p^n}$, 
denoted $\End_n$, is known and is given by (see \cite{RA2})
\[\End_n =  W\mathbf{F}_{p^n}\langle S \rangle /(S^n=p, Sw = w^{\sigma}S) .\]
Here one can think of $S$ as a non-commuting indeterminant.  

We will think of $\End_n$ as a topological monoid under multiplication.   The submonoid consisting of invertible elements is the Morava stabilizer group $S_n = (\End_n)^{\times}$.   Also, $\text{Gal}$ acts on $\End_n$ and hence on $S_n$, and we can form the semidirect product $G_n=S_n \rtimes \text{Gal}$, sometimes referred to as the extended stabilizer group (see for example \cite{DH04}). 

Morava $E$-theory, also referred to as Lubin-Tate theory, is a Landweber exact homology theory represented by a spectrum denoted $E_n$ and corresponding to the  Hopf algebroid

\[(E_{n*},E_{n*}E_{n})  =  (E_{n*},E_{n*}\otimes_{BP_*} BP_*BP \otimes_{BP_*} E_{n*})   .\]
The completion of this Hopf algebroid is
\[(E_{n*},\text{Map}_c(G_n , Z_p)\hat{\otimes}_{Z_p} E_{n*})\]
which we will talk about in the next section (Proposition 2.2 of  \cite{DH04}).
Here $\text{Map}_{c}$ refers to the set of continuous maps. 
The coefficient ring has the following description:  
\[E_{n*} = W\mathbf{F}_{p^n}[[u_1,\dots,u_{n-1}]][u,u^{-1}].\]
The ring $E_{n*}$ is graded by $|u_i| = 0$ and $|u| = -2$.  
There is a graded map of coefficients $BP_*\xrightarrow{\lambda} E_{n*}$ given by 
\begin{equation}\label{coefficient-map}\lambda(v_i) = 
     \begin{cases} u_iu^{1-p^i} &  i<n \\ u^{1-p^n} &  i = n \\ 0 & i> n.
     \end{cases}
\end{equation}

We have the Hopf algebroid associated to Morava $K$-theory 
$(K(n)_*, \Sigma(n))$, where $K(n)_* = \mathbf{F}_p[v_n,v_n^{-1}]$ and 
\[\Sigma(n) = K(n)_*\otimes_{BP_*} BP_*BP  \otimes_{BP_*}  K(n)_*\]  (note that $\Sigma(n) \neq K(n)_*(K(n))$.  See \cite{MR1}).  

We consider the composite map of Hopf algebroids
\begin{equation}\label{reduction}
(B(n)_*, \Gamma_{B(n)_*}) \xrightarrow{} (K(n)_*, \Sigma(n) ) \xrightarrow{} ({E_{n}}_{*}/I_{n},    {E_{n}}_{*}E_{n}/I_{n}) 
 \end{equation}
 We wish to show these are all in the category $\mathcal{H}$. 
 
 First we need to establish a fact about the Hopf ring for $P(n)_* = BP_*/I_{n}$. 
 Again we are grateful for help from Martin Bendersky who suggested the use of the Boardman basis for $QBP_*(\underline{BP}_*)$ \cite{BOAR} and the use of Ravenel and Wilson's calculation of the Hopf ring for $P(n)$ in \cite{RW2}.

We start with the Hopf ring $BP_*(\underline{BP}_*)$ and tensor on the right and the left with $P(n)_*$.  This gives an algebraically defined Hopf ring which corresponds to the factors described on page 3 of \cite{RW2}  for which $a^I =1$.  We further simplify by only considering the even dimensional spaces and the indecomposables.  The result can be extended to the odd dimensional spaces and the primitives by standard arguments.  Denote this object by $QP(n)^{*}_{*}$.  Consider the Boardman basis for the indecomposables $QBP_*(\underline{BP}_*)$, which consists of monomials $v^Ih^{J'}$ and {\em excluding} all monomials of the form $v_{d_0}^{}v_{d_1}^p v_{d_2}^{p^2}\cdots v_{d_l}^{p^l} h_l$, $l\ge 0$, $d_{0} \le d_{1} \le \dots \le d_{l}$,  and any monomial divisible by one of this form. By \cite{BOAR} this is a basis for  $QBP_*(\underline{BP}_*)$ as a {\em right} $BP_*$-module.

\begin{theorem}\label{Boardman-basis-for-P(n)}  The image of the Boardman basis in $QP(n)^{*}_{*}$ is a basis for $QP(n)^{*}_{*}$ as a free right $P(n)_*$-module.   The monomials of the form $v^Ih^{J'}[v^K]$, where $v^I$ and $v^K$ are in $P(n)_*$ and $v^Ih^{J'}$ satisfies the Boardman condition above, give a basis for $QP(n)^{*}_{*}$ as an $F_p$-vector space. 
\end{theorem}

\begin{proof}  We will follow the exposition and results of Boardman \cite{BOAR}, particularly pages 10-12.    Let $w_k$ denote $[v_k]$, and let $W_n$ be the set of monomials in $\{w_n, w_{n+1}, \dots \}$. Similarly $V_n$ will denote monomials in $\{v_n,v_{n+1}, \dots \}$ and $H_m$ will be monomials in $\{h_m,h_{m+1}, \dots\}$.    Boardman defines the Poincare series of a monomial 
$x\in BP_{2i}(\underline{BP}_{2n})$ by $P(x) = s^{i}t^{i-n}$.   For example
\[P(v_i)=s^{p^i-1}t^{p^i-1},\,\,\,P(h_j)=s^{p^j}t^{p^j-1},\,\,\,P(w_k)=t^{p^k-1},\,\,\,P(xy) = P(x)P(y)\]
and for a family $S$, he defines $P(S)$ to be $\Sigma_{x\in S}P(x)$.  For families $S$ and $T$, we have $P(ST)=P(S)P(T)$. 

Boardman observes the formulas 
\[
P(V_k) = \Pi_{r=k}^{\infty} (1-P(v_r))^{-1};\quad
P(H_m) = \Pi_{r=m}^{\infty} (1-P(h_r))^{-1};\quad\text{etc.} 
\]
which are useful. 

The family $A_{k,m}\subset V_kH_m$ is defined by excluding monomials of the form $v_{i_m}^{p^m}v_{i_{m+1}}^{p^{m+1}} \cdots v_{i_l}^{p^l} h_l$, $l \ge m$, $i_{m}\le i_{m+1} \le \dots \le i_{l}$, and any multiple of such.   Boardman proves that the Poincare series satisfy
\begin{equation}\label{boardmans-formula}P(A_{k,m}) = P(V_k)P(H_m)P(H_{k+m})^{-1}.\end{equation}

Note that $A_{n,0}$ is the image of the Boardman basis under the map $BP \to P(n)$, i.e. the image of the Boardman basis by modding out by $I_n$ on the left.    So $A_{n,0} \subset V_{n}H_{0}$ and excludes monomials of the form $v_{i_0}^{}v_{i_{1}}^{p} \cdots v_{i_l}^{p^l} h_l$.  Thus $P(A_{n,0}) = P(V_n)P(H_0)P(H_n)^{-1}$.   
It follows that the Poincare series for the free right $P(n)_*$ module on the image of the Boardman basis is given by $P(V_n)P(H_0)P(H_n)^{-1}P(W_n)$.  
We wish to compare this to the Ravenel-Wilson basis given in \cite{RW2}.  

Define $R_{k,m} \subset H_kW_m$ by excluding monomials $h_{j_m}^{p^m}h_{j_{m+1}}^{p^{m+1}} \cdots h_{j_l}^{p^l} w_l$, $l\ge m$, $j_{m}\le j_{m+1} \le \dots \le j_{l}$.   Note that $R_{0,n}$ is the Ravenel-Wilson basis of $n$-allowable monomials which exhibit  $QP(n)^{*}_{*}$ as a free left $P(n)_*$-module.  

We claim
\begin{equation}\label{poincare-ravenel-wilson}P(R_{k,m}) = P(H_k)P(W_m)P(H_{k+m})^{-1}.\end{equation}

Granting this, the Poincare series for $QP(n)^{*}_{*}$ is given by 
\[P(V_n)P(R_{0,n}) = P(V_n)P(H_0)P(H_n)^{-1}P(W_n).\]
This is the same as the Poincare series of the Boardman basis mod $I_n$. We know the Boardman basis spans  $QP(n)^{*}_{*}$ because the map $BP\to P(n)$ is onto. Therefore the image of the Boardman basis is a basis for $QP(n)^{*}_{*}$ as a free right $P(n)_*$-module. 

To prove the claim  (\ref{poincare-ravenel-wilson}), we follow the very same argument given by Boardman in \cite{BOAR} to prove (\ref{boardmans-formula}).  By the same process, decomposing $R_{k,m} \subset H_kW_m$ according to powers of $h_k$, we get a recurrence relation:
\[P(R_{k,m}) = P(h_k)^{p^m}P(R_{k,m+1}) + \bigl(\Sigma_{r=0}^{p^m-1}P(h_k)^r\bigr)P(R_{k+1,m})\]
and $P(R_{k,m})$ is the unique solution. 

Normalizing, define 
\[f_{k,m} = P(R_{k,m})P(H_k)^{-1}P(W_m)^{-1}.\]
We get 
\[f_{k,m} = P(h_k)^{p^m}f_{k,m+1}(1-P(w_m)) + \bigl(\Sigma_{r=0}^{p^m-1}P(h_k)^r\bigr)f_{k+1,m}(1-P(h_k))\]
which becomes 
\[f_{k,m} = P(h_k)^{p^m}f_{k,m+1}(1-P(w_m)) + (1-P(h_k)^{p^m})f_{k+1,m}.\]

Now let $f_{k,m} = P(H_{k+m})^{-1}$. Substituting into the recurrence relation  we get 

\begin{multline*}   P(H_{k+m})^{-1} =   P(h_k)^{p^m}P(H_{k+m+1})^{-1}(1-P(w_m)) + (1-P(h_k)^{p^m})P(H_{k+m+1})^{-1}\\
\dfrac{P(H_{k+m})^{-1}}{P(H_{k+m+1})^{-1}} =   P(h_k)^{p^m}(1-P(w_m)) + 1 - P(h_k)^{p^m} \\
1-P(h_{k+m}) = 1 -  P(h_k)^{p^m}P(w_m).
 \end{multline*}
This is true because $P(h_j) = s^{p^j}t^{p^j-1}$ and $P(w_j)=t^{p^j-1}$ by definition.

\end{proof}

 \begin{proposition}\label{exactness-of-U-non-Landweber-exact-case}
 The Hopf algebroids in \ref{reduction} are in the category $\mathcal{H}$. 
 \end{proposition}
 
 \begin{proof}
The flatness condition in Definition \ref{unstable-comodule-categories} is easy to check.  The exactness condition is harder.  Since all $K(n)_*$ and ${E_{n}}_{*}/I_{n}$-modules are free it is immediate that $U$ is exact in those cases.  
Now consider $(B,\Sigma) = (P(n)_*,  BP_*BP/I_{n})$.   The proof is essentially the same proof as in Proposition \ref{exactness-of-U-Landweber-exact-case}, using Theorem \ref{Boardman-basis-for-P(n)}, with the obvious modification of changing free $Z_{(p)}$-modules to $F_{p}$-vector spaces.   

Finally, $B=B(n)_*$ is obtained from $P(n)_*$ by inverting $v_n$, and since direct limits preserve exactness the result follows for $B(n)_*$ as well.  
 
\end{proof}

The following result is the first part of the proof of Theorem \ref{first-main-theorem}.

\begin{theorem}\label{first-part-of-first-main-theorem}  
Using the notation of Theorem \ref{first-main-theorem}, there is an isomorphism
\[    
\ext^{s}_{\mathcal{U}_{    \Gamma_{B}   }}(B[t],M)  
\cong \ext^{s}_{\mathcal{U}_{    {E_{n}}_{*}E_{n}/I_{n}  }}({E_{n}}_{*}/I_{n}[t],{E_{n}}_{*}/I_{n}\otimes_{B(n)_{*}} M      ).
\]
\end{theorem}

\begin{proof}    
It is proved in \cite{HOVSA} using an observation of N. Strickland regarding a result of Lazard's (see Theorem 3.4 and the proof of Theorem 3.1 there)  that the faithfully flat condition of Theorem \ref{unstable-faithflat-equiv} is satisfied for the first map in (\ref{reduction}).   The second map 
\[K(n)_*   \xrightarrow{}  {E_{n}}_{*}/I_{n}  \]
is a faithfully flat extension, so again Theorem  \ref{unstable-faithflat-equiv} applies.

\end{proof}

\noindent {\bf Remark}. In fact, this result can be generalized to an unstable version of Hovey-Sadofsky's change of rings theorem, since in \cite{HOVSA} they show that 
for $j \le n$ the map 
 \begin{equation}
(B(j)_*, \Gamma_{B(j)_*})  \xrightarrow{} (v_j^{-1}{E({n})}_{*}/I_{j},    v_j^{-1}{E({n})}_{*}E({n})/I_{j}), 
\end{equation}
satisfies the conditions of Theorem of \ref{unstable-faithflat-equiv}, but we will not use that in this paper.

%%%%%%%%%%%%%%%%%%%%%%%%%%%%%%%%%%%%%%%%%%%
%
%  \section{More on Unstable Comodules}
%
%%%%%%%%%%%%%%%%%%%%%%%%%%%%%%%%%%%%%%%%%%%%
\section{More on Unstable Comodules}

Now we give the description of unstable comodules in Morava $E$-theory that we are after.    Start by recalling from \cite{DEV} that there is a Hopf algebroid $(U,US)$ which is equivalent to $(BP_*,BP_*BP)$ and lies between  
$(BP_*,BP_*BP)$ and $(E_{n*},  E_{n*}E_n)$.   

Let $FGL_{p}$ be the groupoid valued functor on graded $p$-local algebras which assigns to an algebra $A$ the groupoid $FGL_{p}(A)$ whose objects are $p$-typical formal group laws over $A$, and whose morphisms are strict isomorphisms (\cite{RA2}).   Let ${FGL_{p}}_{*}$ be the groupoid valued functor on graded $p$-local algebras which assigns to an algebra $A$ the groupoid ${FGL_{p}}_{*}(A)$ whose objects are pairs $(F,a)$ where $F$ is a $p$-typical formal group law over $A$, $a\in A^{\times}$, and a morphism  $f:(F,a) \to (G,b)$  is an isomorphism from $F$ to $G$ with 
$a = f'(0)b$.   If $F$ is a $p$-typical formal group law over $A$, and $a$ is a unit in $A$, define a formal group law $F^{a}$ by $F^{a}(x,y) = a^{-1}F(ax,ay)$.
 
 Define graded algebras,  
\[U=\mathbf{Z}_{(p)}[u_1,u_2,\dots][u,u^{-1}]\] 
and 
\[US = U[s_0^{\pm 1},s_1,s_2,\dots]\]
with $|u_i| = 0$, for $i\ge 1$, $|s_i| = 0$ for $i \ge 0$, and $|u|= -2$. 

There is a Hopf algebroid structure on $(U,US)$. There are isomorphisms of groupoid schemes 
\begin{align*}
\theta:(\spec BP_*,\spec BP_*BP) &\to FGL_{p}^{\text{opp}}\\
\theta_{*}:(\spec U,\spec US) &\to {FGL_{p}}_{*}^{\text{opp}}
\end{align*}
and a natural transformation 
\[(\spec U , \spec US ) \xrightarrow{\spec \lambda} (\spec BP_* , \spec BP_*BP ),\]
 which sends $(F,a)$ to the formal group law $F^{a^{-1}}$. This natural transformation is represented by the graded Hopf algebroid map 
 \[(BP_*, BP_*BP) \xrightarrow{\lambda} (U,US)\]
 given by
\begin{align*}
\lambda(v_i) &= u_i u^{-(p^i-1)} \\
\lambda(t_i)  &= s_i u^{-(p^i-1)} s_0^{-p^i}\
\end{align*}
See \cite{DEV} for more details and proofs of these assertions.

The map ${\lambda}$ is a faithfully flat extension of coefficient rings, hence by Hopkins' theorem induces an equivalence of comodule categories.  We want to identify the unstable comodule category.  Unstably it is preferable to use the generators for $BP_*BP$ given by $h_i = c(t_i)$ where $c$ is the canonical antiautomorphism.   In $US$ define $c_i= c(s_i)$ and note that 
\[c_0 = c(s_0) = s_0^{-1}\] and 
\begin{equation*}
\eta_{R}(u) = c(\eta_{L}(u)) = s_0u = c_0^{-1} u.
\end{equation*}

Morava $E$-theory is obtained from $(U,US)$ by killing off $u_i$ for $i>n$, setting $u_n=1$, completing with respect to the ideal $I_n=(p,u_1,u_2,\dots, u_{n-1})$, and tensoring with the 
Witt ring $W\mathbf{F}_{p^n}$.  We have 
\[(E_{n*},E_{n*}E_n) = (E_{n*},E_{n*}[c_0^{\pm 1},c_1,\dots]{\otimes}_{U} E_{n*}).\]
If we reduce modulo  $I_n$ then 
\begin{multline*}       
  (E_{n*}/I_{n},E_{n*}E_n/I_{n})   \\  
  = (\mathbf{F}_{p^n}[u,u^{-1}], \mathbf{F}_{p^n}[u,u^{-1}][c_0^{\pm 1},c_1,\dots]    {\otimes}_{U} E_{n*}        ).
\end{multline*}

Applying the canonical anti-automorphism $c$ to the map $\lambda$ we get 

\begin{align*}
 \lambda(h_i) &= c_i (s_0u)^{-(p^i-1)}c(s_0)^{-p^i} \\
 &= c_i s_0^{-(p^i-1)}u^{-(p^i-1)}s_0^{p^i} \\
 &= c_i u^{-(p^i-1)} s_0 \\
 &= c_i u^{-(p^i-1)} c_0^{-1} 
\end{align*}

Let $K=(k_1,k_2,\dots)$ be a finite sequence of non-negative integers and denote $h_{1}^{k_1}h_{2}^{k_2}\dots$ by $h^K$ and similiarly $c_{1}^{k_1}c_{2}^{k_2}\dots$ by $c^K$.  Also denote \[|K|=k_{1}(p-1)+k_{2}(p^2-1) + \dots\] and \[l(K) = k_1 + k_2 + \dots.\]  Then we have
\[\lambda(h^K) = c^K u^{-|K|}c_0^{-l(K)}.\]

 If $M$ is a $(BP_*,BP_*BP)$-comodule with coaction 
\[M \xrightarrow{\psi} BP_*BP \otimes_{BP_*} M\]
then for each $x\in M$ we have
\[  \psi(x) = \sum_K v_{K}h^{K} \otimes m_{K}\]
where the sum is indexed over sequences $K$.  The coefficient $v_{K}$ is just some element in $BP_*$.  For each term in the sum we make the following calculation.  Assume $m_{K}$ is even.

\begin{align*} 
\label{linchpin}\lambda(v_{K}h^{K}) \otimes m_{K}  &= u_{K}u^{-|v_K|/2} c^{K}u^{-|K|}c_0^{-l(K)} \otimes m_{K}u^{|m_{K}|/2} u^{-|m_{K}|/2} \\
&= u_{K}u^{-|v_K|/2}  c^{K}u^{-|K|}c_0^{-l(K)}\eta_{R}(u^{-|m_{K}|/2}) \otimes m_{K}u^{|m_{K}|/2} \\
&= u_{K}u^{-|v_K|/2}  c^{K}u^{-|K|}c_0^{-l(K)}(s_0u)^{-|m_{K}|/2} \otimes m_{K}u^{|m_{K}|/2} \\
&= u_{K}u^{(-|v_K|/2-|m_{K}|/2-|K|)}  c^{K}c_0^{-l(K) +|m_{K}|/2} \otimes m_{K}u^{|m_{K}|/2}\\
&= u_{K}u^{(-|v_K|/2-|m_{K}|/2-|K|)}  c^{K}c_0^{-l(K) +|m_{K}|/2}\otimes y
\end{align*}
where $|y|=0$.  In the case where $|m_k|$ is odd, multiply and divide on the right by $u^{(|m_{K}|-1)/2}$ resulting in $y$ on the right with $|y|=1$. 
This motivates the following definition. 

 \begin{definition}\label{non-negative-comodule}    Let $(A,\Gamma)$ denote either the Hopf algebroid $(U,US)$ or 
 $(E_{n*}, E_{n*}E_n)$.   Suppose $M$ is a free $A$-module.  Define $V_{\ge 0}(M)$ to be the sub-$A$-module of $\Gamma \otimes_{A} M$ spanned by elements of the form $\gamma \otimes y$ where $y$ is in degree $0$ or $1$, and 
$\gamma = c_0^{k_0}c^K$ with $k_0 \ge 0$.    
 Then $V_{\ge 0}$ defines an endofunctor on the category of free $A$-modules.  Extend this to an endofunctor on all $A$-modules as in Definition \ref{extend-U}.   Now $V_{\ge 0}$ is the functor of a comonad on $A$-modules. Call the coalgebras over $V_{\ge 0}$ the non-negative comodules.   
 
\end{definition}

Recall the category of unstable comodules $\mathcal{V}_{\Gamma}$ defined in \ref{unstable-comodule-categories-Landweber-exact}.  
\begin{proposition}
\label{non-negative-comodules}  
The categories $\mathcal{V}_{BP_*BP}$ and $\mathcal{V}_{US}$  are both equivalent to the category of non-negative $US$-comodules.  The category  $\mathcal{V}_{E_{n*}E_n}$ is equivalent to the category of non-negative $E_{n*}E_n$-comodules. 
\end{proposition}

\begin{proof}  For $BP$, by definition a $BP_*BP$-comodule that is free as a $BP_*$-module is unstable if the coaction on each element is in the $BP_*$-span of elements of the form $h^K\otimes m_K$ where 
$2l(K) \le |m_K|$.   For $(U,US)$ and $E_{n*}(E_n)$ the same condition applies, using the generators $b_i = c_i u^{-(p^i-1)} c_0^{-1}$ which are the images under $\lambda$ of $h_i$.  (Refer to \cite{BCM} or  \cite{BH}.)
The calculation above shows that $V_{\Gamma}(M) = V_{\ge 0}(M)$ for every $M$ that is free as an $A$-module.   Since this is true on free modules, it is true on all $A$-modules, and the conclusion for $(A,\Gamma)$ follows.  Also, $\mathcal{V}_{BP_*BP}$ is equivalent to $\mathcal{V}_{US}$ by Theorem \ref{unstable-faithflat-equiv} because the map $BP_*\to U$ is faithfully flat.
\end{proof}

\noindent {\bf Remark}.  This is an analog for height $n$ of a height $\infty$ result which is described in \cite{POW} (Theorem 4.1.4)  and \cite{BJ} (Section 4 and Appendix B).  It is classical that the dual Steenrod algebra is a group scheme which represents the automorphism group of the additive formal group law.   If one considers endomorphisms of the additive formal group law, not necessarily invertible, the representing object of this monoid scheme is a bialgebra, i.e. a 'Hopf algebra without an antiautomorphism'.  At the prime $2$ this is described explicitly in \cite{BJ} (see Section 4 and Appendix B).  Whereas the classical dual Steenrod is expressed as $\mathcal{S}_* = \mathbf{Z}/2[\xi_1,\xi_2,\dots]$,  the extended Milnor coalgebra  is 
$\mathcal{A} = \mathbf{Z}/2[a_{0}^{\pm 1},a_1,a_2,\dots]$.   One can see that there is an equivalence between the category of graded comodules over $\mathcal{S}_*$ and the category of comodules over $\mathcal{A}$.   One can also see that under this equivalence, the category of graded unstable comodules over $\mathcal{S}_*$ is equivalent to the category of 'positive' $\mathcal{A}$-comodules, i.e. comodules over the bialgebra $\mathcal{A}^{+} =  \mathbf{Z}/2[a_{0},a_1,a_2,\dots]$.   In \cite{POW} this result is extended to odd primes and generalized.
Our Proposition \ref{non-negative-comodules} is a version for the Landweber-Novikov algebra.  This goes back to \cite{MOR}.

We want to translate this to the 
$\mathbf{Z}/2$-graded case.   First consider what happens stably.  
Let $M$ be an $E_{n*}E_{n}$-comodule, and for simplicity assume for the moment $M$ is concentrated in even degrees. 

Let $((E_{n})_{0},(E_{n})_0(E_n))$ be the Hopf Algebroid of elements in degree $0$.  There is functor $M \mapsto M_{0}$ from the category of $E_{n*}E_{n}$-comodules to the category of $(E_{n})_0E_{n}$-comodules defined as follows:
As an  $(E_{n})_0$-module let $M_{0}$ be the elements in $M$ of degree $0$.  For $x\in M_{0}$ suppose the $E_{n*}E_{n}$-coproduct is given by $\psi_{M}(x) = \sum \gamma_i \otimes x_i$ and define the ungraded coproduct by 
\[ \psi_{M_{0}}(x) = \sum \gamma_i \eta_{R}(u^{-|x_i|/2})\otimes u^{|x_i|/2}x_i.\]
This defines an equivalence of categories from $E_{n*}E_{n}$-comodules to $(E_{n})_{0}E_{n}$-comodules and an isomorphism of Ext groups 
\[ \ext^{s,2t}_{E_{n*}E_{n}}(A,M) \xrightarrow{\cong}   \ext^{s,0}_{E_{n*}E_{n}}(A[2t],M) \xrightarrow{\cong} 
\ext^{s}_{(E_{n})_0(E_{n})}(A[2t]_0, M_{0}).\]

There is an analogous statement for comodules with elements in odd degrees.  Combining the two cases we get the functor $M \to M_0 \oplus M_1$ to $\mathbf{Z}/2$-graded $(E_{n})_{0}E_{n}$-comodules.

Looking at $\mathbf{Z}/2$-graded, unstable comodules over the Hopf algebra
\begin{equation*}
  ((E_{n})_{0}/I_{n},(E_{n})_{0}E_n/I_{n})  
  = (\mathbf{F}_{p^n}, \mathbf{F}_{p^n}[c_0^{\pm 1},c_1,\dots]/   (c_i^{p^{n}}-c_i)        ),
\end{equation*}
denote this Hopf algebra by $(B_0,\Sigma_0)$.    The calculation above suggests that we consider the algebra 
$\mathbf{F}_{p^n}[c_0,c_1,\dots]/   (c_i^{p^{n}}-c_i)$.    The coproduct preserves the non-negativity of the exponent of $c_0$ and so this is a bialgebra.  It is not a Hopf algebra as there is no anti-automorphism.   Soon we'll identify this bialgebra explicitly as a co-monoid object in the category of algebras.     

\vspace{.2cm}

\noindent {\bf Remark}. The map
\[
\mathbf{F}_{p^n}[c_0,c_1,\dots]/   (c_i^{p^{n}}-c_i) \xrightarrow{\sigma'} \mathbf{F}_{p^n}[c_0^{\pm 1},c_1,\dots]/   (c_i^{p^{n}}-c_i),
\]
which corresponds to suspension, is not an injection because $c_{0}^{p^n-1} -1$ is in the kernel.  For example, consider the case $n=1$. The module  
$U_{E(1)_*E(1)}(E(1)_*(S^1)/p)$ contains the element $1 \otimes v_1\iota_1  - v_1\otimes \iota_1$ which is non-zero even though $\eta_R(v_1) = v_1 - ph_1$ because $h_1$ doesn't live on the $1$-sphere.   However this element suspends to zero in $E(1)_*E(1)\otimes_{E(1)_*} E(1)_*(S^1)/p$.  There is a map 
\[  U_{E(1)_*E(1)}(E(1)_*(S^1)/p) \xrightarrow{} \mathbf{F}_{p}[c_0,c_1,\dots]/   (c_i^{p}-c_i)\]
(see Proposition \ref{bialgebra-comonad} below) and this element goes to $c_{0}^{p-1} -1$.

Let  $\Sigma = {E_{n}}_{*}E_n/I_{n}$  and let $N$ denote an unstable $\Sigma$-comodule.%, abbreviate ${E_{n}}_{*}/I_{n}\otimes_{B(n)_{*}} M$ by $N$.

\begin{proposition}\label{bialgebra-comonad}  There are isomorphisms
\begin{align*}
\ext^{s}_{\mathcal{V}_{   \Sigma  }}&({E_{n}}_{*}/I_{n}[t], N  )            \\
& \cong    \ext^{s}_{     \mathcal{V}_{   \Sigma_{0}   }          }   (   (E_{n})_{0}/I_{n}[t] \oplus
  (E_{n})_{1}/I_{n}[t]    ,  N_{0}\oplus N_{1}      )       \\
&{\cong} 
 \ext^{s}_{                 \mathbf{F}_{p^n}[c_0,c_1,\dots]/   (c_i^{p^{n}}-c_i)           }(   (E_{n})_{0}/I_{n}[t] \oplus
  (E_{n})_{1}/I_{n}[t] ,N_{0}\oplus N_{1}).
 \end{align*}
 \end{proposition}
 
 \begin{proof}  We need to show that the comonad ${V}_{\Sigma_{0}}$ on the category of $\mathbf{Z}/2$-graded $\mathbf{F}_{p^n}$-modules is isomorphic to the comonad given by tensoring with the bialgebra $\mathbf{F}_{p^n}[c_0,c_1,\dots]/   (c_i^{p^{n}}-c_i)$.  
 
 It suffices to consider an $\mathbf{F}_{p^n}$-module $N_0$ which has rank one and is assumed to be in degree zero.  The degree one case is similar. Recall $V_{\Sigma_0}(N_0)$ is defined as a quotient of $V_{({E_n})_{0}E}(M)$, where $M$ is a free 
 ${(E_n)}_0$-module which maps onto $N_0$, as depicted in the following diagram.  We abbreviate $((E_n)_{0},(E_n)_{0}E_n)$ to $(A_{0},\Gamma_0)$ and $\mathbf{F}_{p^n}[c_0,c_1,\dots]/   (c_i^{p^{n}}-c_i)$ to 
 $\mathbf{F}_{p^n}[c_0,c_1,\dots]/{\sim}$.

\begin{center}
\begin{tikzpicture}
\node (a) at (0,0) {$\Gamma_0 \otimes_{A_0} M$};
\node (b) at (3,0) {$\Gamma_0 \otimes_{A_0} N_0$};
\node (c) at (6,0) {$\Sigma_0 \otimes_{A_0}  N_0$};
\node (d) at (10,0) {$\mathbf{F}_{p^n}[c_0^{\pm 1},c_1,\dots]/{\sim}\otimes N_0$};
\node (e) at (0,-2) {$V_{\Gamma_0}(M)$};
\node (f) at (3,-2) {$V_{\Gamma_0}(N_0)$};
\node (g) at (6,-2) {$B_{0}\otimes_{A_0}      V_{\Gamma_0}(N_0)  $};
\node (h) at (10,-2) {$\mathbf{F}_{p^n}[c_0,c_1,\dots]/{\sim}\otimes N_0$};
\path[->,font=\scriptsize,>=angle 90]
(a) edge node[above] {} (b)
(b) edge node[above] {$=$} (c)
(c) edge node[above] {$=$} (d)
(e) edge node[above] {} (f)
(f) edge node[above] {} (g)
(g) edge node[above] {$\pi$} (h);
\path[right hook->,font=\scriptsize,>=angle 90]
(e) edge node[above] {} (a);
\path[->,font=\scriptsize,>=angle 90]
(f) edge node[above] {} (b)
(g) edge node[left] {$\sigma$} (c)
(h) edge node[left] {$\sigma'$} (d);
\node (k) at (6,-4) {$   V_{\Sigma_0}(N_0)  $};
\path[->,font=\scriptsize,>=angle 90]
(k) edge node[left] {$=$} (g);
\end{tikzpicture}
\end{center}

The leftmost vertical map is an injection by definition, since $M$ is a free $A_{0}$-module.  The middle  vertical maps are not injections because $N_0$ has torsion.   The top middle horizontal map is an isomorphism because stably $I_n$ is an invariant ideal. By Proposition \ref{non-negative-comodules}  the stabilization map 
\[B_{0}\otimes_{A_0} V_{\Gamma_0}(N_0) \xrightarrow{\sigma} \mathbf{F}_{p^n}[c_0^{\pm 1},c_1,\dots]/{\sim}\otimes N_0\]
factors through a surjective map $\pi$.    We need to show that $\pi$ is injective.  It is sufficient to show that 
$\pi\vert_{\ker{\sigma}}:\ker{\sigma} \to \ker{\sigma'}$ is injective.   The kernel of $\sigma'$ has an $\mathbf{F}_{p^n}$-vector space basis consisting of the set of monomials $\mathcal{B}' = \{(c_{0}^{p^n-1} - 1)c^K\}$ indexed by finite non-negative sequences $K=(k_1,k_2,\dots)$ satsifying $k_i < p^n$.   The corresponding set of elements $\mathcal{B} = \{(c_{0}^{p^n-1}-1)c^K\}$ in $V_{\Gamma_0}(M)$ maps bijectively to $\mathcal{B}'$, and so the vector space span of the image of $\mathcal{B}$ in $V_{\Sigma_0}(N_0)$, call it $S$, is mapped isomorphically to $\ker{\sigma'}$.  We just need to check that $S$ is all of $\ker{\sigma}$.   If $x\in \ker{\sigma}$, write $x$ as the image of an element $y\in V_{\Gamma_0}(M)$.  We can assume that $y$ is a polynomial in the $c_i$'s with coefficients in the Witt ring since any terms containing elements in $I_n$ will map to zero in $V_{\Sigma_0}(N_0)$. Furthermore we can take the exponent of $c_0$ to be non-negative since $y$ satisfies the unstable condition.  Finally, we claim that all the exponents of the $c_i$'s in $y$ can be taken to satisfy $k_i < p^n$, because $c_i^{p^n} = c_i$ holds in $V_{\Sigma_0}(N_0)$.   Then, since by assumption the image of $y$ in $\Sigma_0 \otimes_{A_0}  N_0$ is zero and hence is in the span of $\mathcal{B'}$, this implies that $x \in S$. 

In order to verify the claim, we examine the relation $c_{i}^{p^n} = c_i$
which comes from  Ravenel's formula (A2.2.5 of \cite{RA2}):

 \begin{equation*}
 \Sigma^{F^*} h_{k}v_{j}^{p^k} = \Sigma^{F^*} h_{k}^{p^j}\eta_{R}(v_j).
 \end{equation*}
 This version of the formula is obtained by applying the Hopf Algebroid anti-automorphism to the formula in \cite{RA2}.
 The coefficients of the formal sum are polynomials in the $\eta_{R}(v)$'s.  We are interested in the relation $r$ which is given by the %mod $I_n$ 
terms in dimension $|v_{n+i}|$.   It is readily verified by the unstable condition of Definition \ref{original-definition} that $r\otimes m$ is defined in $V$ when $|m|=2$.  This observation was made by Bendersky and is used in much of his work.

Recall that we are thinking of $N_0$ as the component in degree $0$ of a graded $\Sigma$-comodule $N$.  Let $\iota$ denote an $F_{p^n}$-module generator of $N_0$.   Then $u^{-1}\iota$ is a generator in dimension $2$.  We get 
\begin{align*}
&(v_{n+ i} + h_1v_{n + i -1}^{p} +h_{2}v_{n+i-2}^{p^2} + \dots + h_{n+i}p^{p^{n+i}}) \otimes   u^{-1}\iota\\
=&(\eta_{R}(v_{n+i})+h_{1}^{p^{n+i-1}}\eta_{R}(v_{n+i-1})+\dots +h_{n+i}p) \otimes   u^{-1}\iota\\
&+(-(\Sigma^{F^*}_{k+j < n+i}h_{k}v_{j}^{p^k})_{|v_{n+i}|} +(\Sigma^{F^*}_{k+j < n+i}h_{k}^{p^j}\eta_{R}(v_{j}))_{|v_{n+i}|})  \otimes   u^{-1}\iota
\end{align*}
The subscript $|v_{n+i}|$ in the third set refers to the terms in the sum that are in the appropriate dimension.

First consider $i=0$.   In $V_{\Sigma}(N)$ we have $v_{n}\otimes u^{-1}\iota = 1\otimes v_{n}u^{-1}\iota$. 
Then, using $v_n = u^{-(p^n-1)}$ and $\eta_{R}(u^{-1}) = u^{-1}c_0$, we get
\begin{align*}
u^{-(p^n-1)}\otimes u^{-1}\iota &= 1\otimes u^{-(p^n-1)}u^{-1}\iota\\
u^{-(p^n-1)}\eta_{R}(u^{-1})\otimes \iota &= \eta_{R}(u^{-p^n})\otimes \iota\\
u^{-p^n}c_0 \otimes \iota &= u^{-p^n}c_0^{p^n} \otimes \iota.
\end{align*}
Since multiplication by $u$ is an isomorphism  $N_{i} \cong N_{i-2}$ we conclude 

$c_0^{p^n} = c_0$ holds in $V_{\Sigma_0}(N_0)$. 

Now suppose $i>0$.   By 6.1.13 of \cite{RA2}  in $V_{\Sigma}(N)$ we get
\[h_{i}v_{n}^{p^i}\otimes u^{-1}\iota = h_i^{p^n}\eta_{R}(v_n) \otimes u^{-1}\iota.\]
Substituting $c_iu^{-(p^i-1)}c_0^{-1}$ for $h_i$,  and $u^{-(p^n-1)}$ for $v_n$,  we get $c_i^{p^n} \otimes \iota = c_i \otimes \iota$, 
and we conclude that $c_i^{p^n} = c_i$ holds in $V_{\Sigma_0}(N_0)$.

 \end{proof}

The next step is to interpret an unstable $(E_{n})_0(E_n)$-comodule in terms of a continuous action of the monoid $\End_n$.
The Galois group $\Gal$ acts on $E_{n*}$ by acting on the Witt ring and we have 
\break $E_{n*}^{\Gal} = \mathbf{Z}_p[[u_1,\dots,u_{n-1}]][u,u^{-1}]$.  (In \cite{DEV} this ring is denoted $E_{*}\hat{}$.)
According to Morava theory, after completing, there is an isomorphism of Hopf algebroids (Theorem 2.1 of \cite{DH04})

\begin{equation}\label{mapping-space}
(E_{n*}^{\Gal} , (E_{n*}^{\Gal} E_{n}^{\Gal} )_{I_n}^{\hat{}})\cong (E_{n*}^{\Gal} ,\text{Map}_c(S_n , W\mathbf{F}_{p^n})^{\Gal}\hat{\otimes}_{\mathbf{Z}_p} E_{n*}^{\Gal} ).
\end{equation}
The category of graded, complete comodules over this Hopf algebroid is equivalent to the category of continuous, filtered, Galois equivariant twisted $S_n-E_{n*}$ modules. See \cite{DEV}, \cite{DH04} for details.   

Mod $I_n$, in degree zero, the Hopf algebroid of equation \ref{mapping-space} becomes
\[\label{stable-mod-p-mapping-space}(\mathbf{F}_{p},\text{Map}_c(S_n , \mathbf{F}_{p^n})^{\Gal})
=(\mathbf{F}_{p},\mathbf{F}_{p}[c_0^{\pm 1},c_1,\dots]/(c_i^{p^n}-c_i)).\]

The explicit description of the group scheme of automorphisms of $\Gamma_n$ over an $\mathbf{F}_{p}$-algebra $k$ is as follows.   Let 
$D=\mathbf{F}_{p}[c_0^{\pm 1},c_1,\dots]/(c_i^{p^n}-c_i)$. 
In \cite{RA2} it is shown that every endomorphism of $\Gamma_n$, i.e. a power series $f$ satisfying
\[f(\Gamma_n(x,y)) = \Gamma_n(f(x),f(y)),\]  
has the form 
\[f(x)  = \sum_{i\ge 0}{}^{\Gamma_n} a_i x^{p^{i}}, \quad a_i\in k \] 
and this will be an automorphism if and only if $a_0 \in k^{\times}$.

For a ring map $h:D\to k$ let $h$ give the automorphism 
\[f(x)  = \sum_{i\ge 0}{}^{\Gamma_n} h(c_i) x^{p^{i}}.\]  
If we do not require the coefficient of $x$ to be a unit, then it is apparent that 
$\spec(\mathbf{F}_{p}[c_0,c_1,\dots]/(c_i^{p^n}-c_i))$  is the monoid scheme whose value on $k$ is the monoid of endomorphisms of $\Gamma_n$ over $k$.

\begin{proposition}\label{unstable-mod-p-mapping-space}
There is an isomorphism of bialgebras
\[
(\mathbf{F}_{p},\Map_c(\End_n , \mathbf{F}_{p^n})^{\Gal})
=(\mathbf{F}_{p},\mathbf{F}_{p}[c_0,c_1,\dots]/(c_i^{p^n}-c_i)).
\]
\end{proposition}

\begin{proof}
The proof given in Section four of \cite{DEV} applies to $\End_n$ as well. In particular equation (4.14) of \cite{DEV} establishes the result one generator at a time. 
\end{proof}

So we are studying discrete left comodules over the discrete bialgebra 
\break $\Map_c(\End_n , \mathbf{F}_{p^n})^{\Gal}$.   Still following \cite{DEV},  given a left comodule $M$ with coaction 
\[ M \xrightarrow{\psi_M}  \Map_c(\End_n , \mathbf{F}_{p^n})^{\Gal} {\otimes} M \cong \Map_c(\End_n , \mathbf{F}_{p^n}    \otimes M         )^{\Gal} \]
define a right action of $\End_n$ on $\mathbf{F}_{p^n}    \otimes M$ by
\[ (a\otimes m)g = a\psi_M(m)(g).\]
Note that this is a right action.

\begin{proposition}
The functor $M \to \mathbf{F}_{p^n} \otimes M$ is an equivalence from the category of $\mathbf{Z}/2$-graded discrete $\Map_c(\End_n , \mathbf{F}_{p^n})^{\Gal}$-comodules to the category of discrete continuous  Galois equivariant right $\End_n$-modules.
\end{proposition}

\begin{proof}
This is the $\mod I_n$, $\End_n$-analog of Proposition 5.3 of \cite{DEV}, and the same proof applies.  The $\End_n$-modules we are considering are modules over 
$\mathbf{F}_{p^n}$. 
\end{proof}

Now we complete the proof of Theorem \ref{first-main-theorem}.
\begin{proof}
Since the $\Gamma$-comodule $M$ is concentrated in odd degrees, 
\[\ext^{s}_{\mathcal{V}_{\Gamma_B}}(B[t], M) \cong \ext^{s}_{\mathcal{U}_{\Gamma_B}}(B[t], M).\]

By Propositions \ref{bialgebra-comonad} and  \ref{unstable-mod-p-mapping-space} it remains to prove that 
\begin{multline*}
\ext^{s}_{\Map_c(\End_n , \mathbf{F}_{p^n})^{\Gal}}((E_{n})_{1}[t]/I_{n} , M)   \\
\cong 
\ext^{s}_{\End_n}(  (E_{n})_{1}[t]/I_{n} , \mathbf{F}_{p^n} \otimes M     )^{\Gal}.
\end{multline*}
for a $(\mathbf{F}_{p},\Map_c(\End_n , \mathbf{F}_{p^n})^{\Gal})$-comodule.   
The group on the right is $\Gal$-equivariant continuous Ext over the monoid $\End_n$.

Again the proof is an adaptation of the proof given in \cite{DEV}.  The cohomology of $\End_n$ with coefficients in a right module $N$ can be defined by the cochain complex 
\[ C^k(\End_n;N) = \Map_c(\End_n \times \dots \times \End_n,N)\] with differential 
\begin{equation*} 
\begin{split}
df(g_1,\dots, g_{k+1}) &= f(g_2, \dots, g_{k+1}) \\
                                     & \quad + \sum_{j=1}^{k}(-1)^{j}(g_1,\dots,g_jg_{j+1},\dots, g_{k+1}) \\
                                     & \quad + (-1)^{k+1}f(g_1,\dots,g_k)g_{k+1}.
\end{split}
\end{equation*}
The cobar complex for $\Map_c(\End_n , \mathbf{F}_{p^n})^{\Gal}$ is isomorphic to $C^*(\End_n;N)^{\Gal}$, the only difference from \cite{DEV} being that we are interpreting the action as a right action. 
\end{proof}

%%%%%%%%%%%%%%%%%%%%%%%%%%%%%%%%%%%%%%%%%%%%%
%
%   \section{Cohomological Dimension}
%
%%%%%%%%%%%%%%%%%%%%%%%%%%%%%%%%%%%%%%%%%%%%%%%
\section{Cohomological Dimension}\label{cohomological-dimension}

The purpose of this section is to prove the following proposition, which implies Theorem \ref{second-main-theorem} of the introduction.

\begin{proposition}\label{morava-vanishing-theorem}  Suppose $M$ be a continuous Galois equivariant \break \mbox{$End_n-(E_{n})_0/I_n$}-module, concentrated in degree $1$. 
Suppose $(p-1) \nmid n$.  Then 
$  \ext^{s}_{\End_n}((E_{n})_{1}[t]/I_{n},M)^{\Gal}     = 0$ for $s > n^2+1$.
\end{proposition}

To begin, note that 
\[
(E_{n})_{1}[t]/I_{n} = \begin{cases}  \mathbf{F}_{p^n} \,\, \text{if $t$ is odd} \\ 0 \,\, \text{if $t$ is even} \end{cases}.
\]
The action of $\End_n$ depends on $t$.  Since 
\[\ext^{s}_{\End_n}(\mathbf{F}_{p^n},M)^{\Gal} \subset \ext^{s}_{\End_n}(\mathbf{F}_{p^n},M)\] 
(see \cite{DEV}), for purposes of studying the vanishing line we can disregard the action of the Galois group.

The proof will be based on a construction used by Bousfield - see for example \cite{BO7},  Subsection 3.1.   Here we carry out a version for the stabilizer group.  To begin, recall from \cite{BO7} that if $E$ is a monoid which possesses an 'absorbing element' $0$, in other words an element such that $0e = e0 = 0$ for all $e\in E$, then any $E$-module $M$ has a decomposition
\[ M = M_{\text{red}} \oplus M_{\text{fix}}\] 
where $M_{\text{red}} = \{ x\in M \,\vert\, x0 =0 \}$ and $M_{\text{fix}} = \{ x\in M \,\vert\, x0 = x \}$.   
Call a module {\em reduced} if $M = M_{\text{red}}$, and {\em trivial} if $M = M_{\text{fix}}$. 
Notice that  $M_{\text{fix}}  = \{ x\in M \,\vert\, xe = x \, \forall e \in E \}$.
Let $\mathcal{E}^{\text{dis}}$ denote the category consisting of discrete $\mathbf{F}_{p^n}$-modules with a continuous right action of $\End_n$, and let $\mathcal{E}^{\text{dis}}_{\text{red}}$ denote the full subcategory of reduced modules. 
Because the functor $M \mapsto M_{\text{red}}$ is right adjoint to the forgetful functor, and the forgetful functor takes monomorphisms to monomorphisms, it follows that if $I$ is injective in $\mathcal{E}^{\text{dis}}$, then $I_{\text{red}}$ is injective in $\mathcal{E}^{\text{dis}}_{\text{red}}$.

The monoid $E=\End_n$ has an absorbing element.  If  
$\mathbf{F}_{p^n}$ is trivial then 
$\ext^{s}_{\End_n}(  \mathbf{F}_{p^n} ,M)    = 0$ for $s>0$ because 
\[\Hom_{\End_n}( \mathbf{F}_{p^n} ,\,\,\,) = (\,\,\,)_{\text{fix}}\] 
is an exact functor. 
So we can assume that $\mathbf{F}_{p^n}$ is reduced, from which it follows that 
\[\ext^{s}_{\mathcal{E}^{\text{dis}}}(  \mathbf{F}_{p^n} ,M) 
=\ext^{s}_{\mathcal{E}^{\text{dis}}_{\text{red}}}(  \mathbf{F}_{p^n} ,M_{\text{red}}).\]
So we can assume without loss of generality that all of our $\End_n$-modules are reduced.

So far we have been considering Ext groups in the category of  discrete $\mathbf{F}_{p^n}$-modules with a continuous action of 
$\End_n$.  The following construction will require us to work in the category of $p$-profinite $\mathbf{F}_{p^n}$-modules with a continuous action of $\End_n$.  Pontryagin duality implies that these two categories of $\End_n$-modules are equivalent.   Note that Pontryagin duality takes $\mathbf{F}_{p^n}$-modules to $\mathbf{F}_{p^n}$-modules, right $\End_n$-modules to left $\End_n$-modules, and reduced modules to reduced modules.

Let $\mathcal{A}$ denote the category consisting of $p$-profinite $\mathbf{F}_{p^n}$-modules with a  continuous left action of 
$S_n$.   Let $\mathcal{E}$ denote the category consisting of $p$-profinite $\mathbf{F}_{p^n}$-modules with a  continuous left action of $\End_n$, and let $\mathcal{E}_{\text{red}}$ denote the full subcategory of 
$\mathcal{E}$ consisting of reduced modules.  There is an obvious forgetful functor $J:\mathcal{E} \to \mathcal{A}$.

\begin{definition}  We define a functor $\tilde{F}:\mathcal{A} \to \mathcal{E}$ as follows:  For an $S_n$-module $M$, let $\tilde{F}(M)$ be $M\times M \times M \dots$ as an abelian group.  For $g\in S_n$, $x=(x_1,x_2,\dots) \in \tilde{F}(M)$, define 
\[gx = (gx_1, g^{\sigma^{-1}}x_2,g^{\sigma^{-2}}x_3,\dots,g^{\sigma^{-(n-1)}}x_n,gx_{n+1}\dots).\]  
For the element $S\in \End_n$, define
$Sx = (0,x_1,x_2,\dots)$.   This defines a continuous $\End_n$ action on $\tilde{F}(M)$ as one can readily check the relation $Sgx = g^{\sigma}Sx$.
\end{definition}

\begin{proposition}\label{Sn-End-adjunction}  The functor $\tilde{F}$ takes values in $\mathcal{E}_{\text{red}}$, and is left adjoint to $J$ restricted to $\mathcal{E}_{\text{red}}$. \end{proposition}

\begin{proof} The unit of the adjunction $M\to J\tilde{F}(M)$ is given by $x\mapsto (x,0,0,\dots)$.  The counit of the adjunction $\tilde{F}J(N) \to N$ is given by 
\[(x_1,x_2,x_3,\dots) \mapsto x_1 + Sx_2 + S^2x_3 + \dots.\]  
which converges because $N$ is reduced.
\end{proof}

Proposition \ref{Sn-End-adjunction} says $\Hom_{\mathcal{E}_{\text{red}}}(\tilde{F}(M),N) \cong \Hom_{\mathcal{A}}(M,JN)$.  We would like a similar statement for $\ext$.

\begin{proposition}  The functors $\tilde{F}$ and $J$ are exact.
It follows that for all~$s$, $\ext^s_{\mathcal{E}_{\text{red}}}(\tilde{F}(M),N) \cong \ext^s_{\mathcal{A}}(M,JN)$.\end{proposition}

\begin{proof}Straightforward.\end{proof}

Now we need a fundamental exact sequence.   For an object $M$ in $\mathcal{A}$ define an object $M'$ in $\mathcal{A}$ as follows.  Let $M' = M$ as $\mathbf{F}_{p^n}$-modules and for each $g\in S_n$, $x'\in M'$, let $gx' = g^{\sigma^{-1}}x$, where $x=x'$ and the expression on the right is the action on $M$.  If $N$ is an object in $\mathcal{E}_{\text{red}}$ there is a map $S:N \to N$.  If we think of $S$ as a map $S:(JN)' \to JN$ then one can check that $S$ is a morphism in $\mathcal{A}$.   Thus we can define $\tilde{F}(S)$.   Also  $S:\tilde{F}((JN)') \to \tilde{F}(JN)$ is a morphism in $\mathcal{E}$.   This gives a map 
\[\partial = \tilde{F}(S)-S:\tilde{F}((JN)') \to \tilde{F}(JN)\] in $\mathcal{E}_{\text{red}}$ and we have

\begin{proposition}  There is a SES in $\mathcal{E}_{\text{red}}$ 
\[0 \xrightarrow{} \tilde{F}((JN)') \xrightarrow{\partial} \tilde{F}(JN) \xrightarrow{} N \xrightarrow{} 0.\]
\end{proposition}

\begin{corollary} There is a LES for any pair of reduced $\End_n$-modules $N$ and $L$. 
\[\dots \to \ext^s_{\mathcal{E}}(N,L) \to \ext^s_{\mathcal{A}}(JN,JL)\to \ext^s_{\mathcal{A}}((JN)',JL)\to   \ext^{s+1}_{\mathcal{E}}(N,L) \to \dots  \]
\end{corollary}

To finish the proof of Proposition \ref{morava-vanishing-theorem} we apply the preceding corollary to the case where $L$ is the Pontryagin dual of $(E_{n})_{1}[t]/I_{n} =  \mathbf{F}_{p^n}$ and $N$ is the Pontryagin dual of $M$.
The forgetful functor $J$, which restricts the action of $\End_n$ to the stabilizer group $S_n$, corresponds to suspension, and $JL$ is dual to $(E_n)_{1}(S^t)/I_{n}$ stably, as an $S_n$-module.   The $S_n$-module $(E_n)_{1}(S^t)/I_{n}$ may or may not have the trivial action, depending on $t$.   However we have an extension
\[1 \to {S'}_n \to S_n \to  \mathbf{F}_{p^n}^{\times} \to 1\]
where ${S'}_n$ is the $p$-Sylow subgroup of $S_n$, i.e. ${S'}_n$ is the group of strict automorphisms of the Honda formal group law, which does act trivially on $(E_n)_{1}(S^t)/I_{n}$,
and the Lyndon-Hochschild-Serre spectral sequence collapses to give 
\[\ext^{0}_{\mathbf{F}_{p^n}^{\times}}((E_{n})_{1}[t]/I_{n},H^s({S'}_n;M))         \cong                 \ext^s_{S_n}((E_{n})_{1}[t]/I_{n},M).\]
Here $H^s({S'}_n;M)$ denotes group cohomology. For the stated values of $n$  and $p$, the $p$-Sylow subgroup ${S'}_n$ of the Morava stabilizer group has finite cohomological dimension equal to $n^2$ ( see for example \cite{MOR} or \cite{RA2})  and \ref{morava-vanishing-theorem} follows.

We will finish this section by sketching an outline of a second possible proof of Proposition \ref{morava-vanishing-theorem}.
This approach, which is conjectural,  because it presumes a Lyndon-Hochschild-Serre spectral sequence in the case of a monoid that is neither discrete nor profinite , may have a more intuitive appeal.

Since we can assume all modules under consideration are reduced, there is an isomorphism 
\[\ext^{s}_{\End_n}(\mathbf{F}_{p^n},M)  = \ext^{s}_{\End_n-\{0\}}(\mathbf{F}_{p^n},M).\]

Any element in $\End_n-\{0\}$ can be uniquely written in the form $gS^k$ where $g\in S_n$ and $k \ge 0$. 
 This gives a monoid isomorphism from $\End_n-\{0\}$ to the semidirect product $S_n \rtimes \mathbf{N}$, where $\mathbf{N}$ is the free monoid on one generator.

Apply the Lyndon-Hochschild-Serre spectral sequence to the extension
\[ 1 \to S_{n} \to S_{n} \rtimes N  \to \mathbf{N} \to 1\]
to get
\[\ext^p_{\mathbf{N}}(\mathbf{F}_{p^n},\ext^{q}_{S_n}(\mathbf{F}_p,M)) \Rightarrow \ext^{p+q}_{S_{n} \rtimes N}(\mathbf{F}_{p^n} ,M).\]
The group $S_{n}$ has finite cohomological dimension equal to $n^2$ as noted above.   Since the cohomological dimension of $\mathbf{N}$ is 1, the result follows.

%%%%%%%%%%%%%%%%%%%%%%%%%%%%%%%%%%%%%%%%%%%%%%%%%%%%%%%%
%
%  The End
%
%%%%%%%%%%%%%%%%%%%%%%%%%%%%%%%%%%%%%%%%%%%%%%%%%%%%%%%

\bibliography{rob}{}
\bibliographystyle{plain}

\end{document}